\documentclass{article}

\usepackage[utf8]{inputenc}
\usepackage{amsmath,amsfonts,amssymb}
\usepackage{amsthm,upref}
\usepackage[a4paper]{geometry}
\usepackage{color} 
\usepackage[dvipsnames]{xcolor}
\usepackage{graphicx}
\usepackage{subfigure}
\usepackage{tabularx,float}
\usepackage[hidelinks]{hyperref}
\usepackage{comment} 
\usepackage{tabularx} 
\usepackage[pagewise]{lineno}
\usepackage{titlefoot} 
\usepackage{fancyhdr}

\newcommand{\note}[1]{ \textcolor{Maroon}{ \ [ \ #1 \  ] \ }} 
\newtheorem{thm}{Theorem}[section]

\newtheorem{proposition}[thm]{Proposition}

\newtheorem{lemma}[thm]{Lemma}
\newtheorem{remark}[thm]{Remark}
\newtheorem{assumption}{Assumption}

\newcommand{\eps}{\varepsilon}
\newcommand{\Diff}{\textup{D}}

\newcommand{\real}{\textrm{Re}\,}

\newcommand{\R}{\mathbb{R}}

\newcommand{\dd}{{\mathrm d}}
\newcommand{\transpose}{\textnormal{T}}
\newcommand{\e}{\mathrm{e}}

\usepackage{todonotes}
\newcommand\rev[1]{{\textcolor{black}{#1}}} 

\title{Rate and Bifurcation Induced Transitions in\\Asymptotically Slow-Fast Systems}
\author{S.~Jelbart\thanks{School of Computation, Information and Technology, Technical University of Munich, Garching, Germany.}}

\begin{document}

	\maketitle
	

	\begin{abstract}
		This work provides a geometric approach to the study of bifurcation and rate induced transitions in a class of non-autonomous systems referred to herein as \textit{asymptotically slow-fast systems}, which may be viewed as `intermediate' between the (smaller resp.~larger) classes of asymptotically autonomous and non-autonomous systems. After showing that the relevant systems can be viewed as singular perturbations of a limiting system with a discontinuity in time, we develop an analytical framework for their analysis based on geometric blow-up techniques. We then provide sufficient conditions for the occurrence of bifurcation and rate induced transitions in low dimensions, as well as sufficient conditions for `tracking' in arbitrary (finite) dimensions, i.e.~the persistence of an attracting and normally hyperbolic manifold through the transitionary regime. The proofs rely on geometric blow-up, a variant of the Melnikov method which applies on non-compact domains, and general invariant manifold theory. The formalism is applicable in arbitrary (finite) dimensions, and for systems with forward and backward attractors characterised by non-trivial (i.e.~non-constant) dependence on time. The results are demonstrated for low dimensional applications.
	\end{abstract}
	
	
	\unmarkedfntext{\textbf{Keywords:} Critical transitions, singular perturbations, geometric blow-up, non-autonomous systems, tipping phenomena}
	
	\unmarkedfntext{\noindent \textbf{MSC2020:} 34E15, 34A26, 37D10, 34C45, 37C60.}

	\section{Introduction}
	\label{sec:introduction}
	
	Bifurcation and rate induced transitions in non-autonomous dynamical systems have received a lot of attention in recent years, due in part to the direct relevance of these systems to important social and economic problems like climate change; see e.g.~\cite{Ashwin2012,Lenton2011,Lenton2008,Scheffer2020,Wieczorek2011} and the many references therein. Significant progress has been made for \textit{asymptotically autonomous systems} of the form
	\begin{equation}
		\label{eq:asymptotically_autonomous_system}
		x' = f(x, \gamma(\mu t), \sigma) , \qquad 
		x \in \R^n ,
	\end{equation}
	where the dash denotes differentiation with respect to time $t$, the function $f : \R^n \times \R \times \Lambda \to \R^n$ is smooth, $\sigma \in \Lambda \subset \R^p$ is a vector of parameters, $\mu > 0$ is the so-called \textit{rate parameter} and $\gamma : \R \to \R$ is a (typically sigmoidal) $C^1$ ramp (a.k.a.~compactification or regularisation) function satisfying
	\[
	\lim_{s \to \pm \infty} \gamma(s) = 
	\begin{cases}
		\gamma_- , & s \to -\infty, \\
		\gamma_+ , & s \to \infty ,
	\end{cases}
	\qquad 
	\lim_{s \to \pm \infty} \gamma'(s) = 0 ,
	\]
	\rev{where $\gamma_\pm$ are constants.} System \eqref{eq:asymptotically_autonomous_system} is called asymptotically autonomous because it converges to `past' and `future' ODE systems that are autonomous as $t \to - \infty$ and $t \to \infty$, namely the systems
	\[
	x' = f(x, \gamma_-, \sigma) , \qquad 
	x' = f(x, \gamma_+, \sigma) ,
	\]
	respectively. This feature is crucial for the application of classical dynamical systems methods, e.g.~center manifold theory and regular perturbation approaches, which can be applied at infinity after a suitable compactification of time.\footnote{\rev{Strictly speaking, this is incorrect. The compactification is actually applied to an auxiliary variable $s = g(t) \in (-1,1)$, for a sufficiently smooth and invertible function $g$, in an extended autonomous system which involves $s$ as an additional dependent variable. The compactification allows for a smooth extension of the extended vector field up to the limiting values $\lim_{t \to \pm \infty} g(t) = \pm 1$.}} In this context, we refer to \cite{Ashwin2017,Ashwin2012,Kuehn2022} for further details and important examples of bifurcation and rate induced transitions in 1-dimensional systems, \cite{Alkhayuon2018} for an example of rate induced transitions between non-equilibrium states of the past and future equations, \rev{\cite{Kiers2020,Wieczorek2023}} for higher dimensional systems and \cite{Chen2023,Wieczorek2023,Wieczorek2021} for details on the compactification of these systems in arbitrary (finite) dimensions (see also \cite{Matsue2023}). Informally, we say that a bifurcation (rate) induced transition occurs if the forward evolution of a particular attractor of the past system under system \eqref{eq:asymptotically_autonomous_system} undergoes an important qualitative change as $t \to \infty$ as a result of varying the bifurcation parameter $\sigma$ (rate parameter $\mu$). We shall also refer to \textit{critical transitions} in cases where we need not specify the underlying mechanism.
	
	Recently, the authors in \cite{Duenas2023c,Duenas2023b,Duenas2023,Longo2022,Longo2023,Longo2021} have considered critical transitions in scalar non-autonomous systems of the form
	\begin{equation}
		\label{eq:nonautonomous_system}
		x' = f(x, \gamma(\mu t), \sigma, t) 
	\end{equation}
	which, in contrast to the asymptotically autonomous system \eqref{eq:asymptotically_autonomous_system}, feature past and future limits which depend non-trivially on time as $t \to \pm \infty$ (note the additional dependence on $t$ in the fourth argument). The additional generality is important for applications which exhibit non-trivial time-dependence before and/or after the `tipping event'. More generally, these works show that a wide variety \rev{of} critical transitions can be defined, identified and classified using the skew-product formalism together with notions from non-autonomous bifurcation theory; we refer to \cite{Anagnostopoulou2023} and the many references therein for background. Detailed analytical treatment of problems within this class remains a highly non-trivial task, and 
	current analyses in this setting have, to the best of the authors knowledge, only been undertaken for scalar equations with additional specific structure. 
	This is due, in part, to the fact that powerful dynamical systems methods that have been developed for autonomous systems cannot in general be applied here, \rev{in contrast to asymptotically autonomous systems \eqref{eq:asymptotically_autonomous_system} described above.} 
	
	Thus, there is a tension between analytical tractability and sufficient generality. 
	A primary aim of this work, is to demonstrate that a compromise which balances tractability with generality can be found by considering an `intermediate' class of non-autonomous systems in the general form
	\begin{equation}
		\label{eq:asymptotically_sf_systems}
		x' = f(x, \gamma (\mu t), \sigma, \eps t, \eps) ,
	\end{equation}
	where $0 < \eps \ll 1$ is a small perturbation parameter. System \eqref{eq:asymptotically_sf_systems} defines a class of dynamical systems which is strictly larger than the class of asymptotically non-autonomous systems \eqref{eq:asymptotically_autonomous_system}, but strictly smaller than the class of non-autonomous systems \eqref{eq:nonautonomous_system}. We shall hereafter refer to systems of the form \eqref{eq:asymptotically_sf_systems} as \textit{asymptotically slow-fast systems}, since it can be shown that the forward and backward limits in time are described by slow-fast systems.\footnote{\rev{This includes but is not limited to the class of slow-fast asymptotically autonomous systems of the form $x' = f(x, \gamma(\mu t), \sigma, \eps)$, for which the dynamics in the inner `transitionary regime' is itself slow-fast as $\eps \to 0$; see e.g.~\cite{O'Sullivan2023} for a recent application within this class.}} The time-scale separation is a consequence of the fact that the `transition time-scale' $t$ over which \rev{$\gamma(\mu t)$} varies from $\gamma(\mu t) \sim \gamma_-$ to $\gamma(\mu t) \sim \gamma_+$, is very fast relative to to the time-scale $\eps t$ associated with time-dependence of the past and future systems. With regard to the tension between tractability and generality outlined above, it is significant to note that for asymptotically slow-fast systems of the form \eqref{eq:asymptotically_sf_systems} we can both (i) incorporate non-trivial time-dependence as $t \to \infty$, and (ii) use established and powerful theory for slow-fast systems \cite{Fenichel1979,Jones1995,Kuehn2015,Wechselberger2019,Wiggins2013} as a rigorous and geometrically informative approach to understanding the past and future dynamics. In this context, we also refer to \cite{Wechselberger2013} for a detailed exposition on the use of \textit{geometric singular perturbation theory (GSPT)} to study slowly non-autonomous systems without a ramp (i.e.~in systems \eqref{eq:asymptotically_autonomous_system} with no $\gamma(\mu t)$-dependence), and to \cite{Kuehn2011,Kuehn2013} for more on the applicability of GSPT in the study of (bifurcation and noise induced) critical transitions.
	
	The primary aim of this work is to make the first steps in the development of a geometric framework for the study of asymptotically slow-fast systems \eqref{eq:asymptotically_sf_systems}. An important feature of these systems is that they can be shown to converge to systems that are \textit{piecewise-smooth (PWS)} in time as $\eps \to 0$. In essence, the time-scale associated with the transition from $\gamma(\mu t) \sim \gamma_-$ to $\gamma(\mu t) \sim \gamma_+$ becomes instantaneous as $\eps \to 0$, and we are left with distinct slow-fast systems on either side of the transition which correspond to the backward and forward limits in time. Of course, this is a very singular limit. We then use the \textit{geometric blow-up method} in order to resolve the loss of smoothness as $\eps \to 0$. This is based on an approach which has been developed in recent years for the study of both regularised PWS systems and smooth systems which converge to PWS systems, see e.g.~\cite{Carvalho2011,Huzak2023,Jelbart2021,Jelbart2020d,Kristiansen2019c,Kristiansen2015a,Kristiansen2019d,Llibre2009}. Here it \rev{is} worthy to note that the blow-up transformation is used to resolve a loss of smoothness, and not \rev{a} loss of hyperbolicity as in the well-known references \cite{Dumortier1996,Krupa2001a,Krupa2001c,Krupa2001b,Szmolyan2001,Szmolyan2004}. In the context of asymptotically slow-fast systems \eqref{eq:asymptotically_autonomous_system}, we shall show that the geometric blow-up naturally incorporates a compactification of the transitionary regime which is similar to that developed in \cite{Chen2023,Wieczorek2023,Wieczorek2021}, as well as providing a geometric way of connecting or `matching' to the outer dynamics of the limiting systems in forward and backward time.
	
	We shall use this approach in order to state sufficient conditions for the transversal intersection of 1-dimensional \rev{locally} invariant manifolds which, for scalar systems \eqref{eq:asymptotically_autonomous_system}, correspond to either bifurcation or rate induced transitions. \rev{Typically, these 1-dimensional manifolds are obtained as \textit{slow manifolds} in the outer regimes, and subsequently extended into the inner regime. After blow-up, the problem of identifying intersections of these manifolds becomes very similar to the problem of identifying \textit{canard solutions} in classical slow-fast systems \cite{DeMaesschalck2021,Dumortier1996,Krupa2001a,Krupa2001b,Szmolyan2001,Wechselberger2012} (where canards arise due to the intersection of attracting and repelling slow manifolds)}, except in the context of systems with a PWS singular limit. A clear analogy to the well-known scenario involving heteroclinic connections in 1-dimensional asymptotically autonomous systems \eqref{eq:asymptotically_autonomous_system} appears (see e.g.~\cite{Ashwin2017,Kuehn2022}), insofar as intersections correspond to the transversal breaking of heteroclinic solutions identified in a nearby (but auxiliary) system which is defined on the blow-up surface. Although we are able to state the result for arbitrary dimensions, 
	`tipping' usually only occurs if the aforementioned heteroclinic connection is codimension-1, i.e.~if it forms a separatrix. In addition to geometric blow-up, our primary tool for the proof is the variant of Melnikov theory which was developed for problems on non-compact domains in \cite{Wechselberger2002}; see also \cite{Krupa2001c} for an earlier formulations in low dimensions which is surprisingly relevant in our setting.
	
	Finally, we provide sufficient conditions for `tracking' or, more precisely, the persistence of an attracting normally hyperbolic manifold through the transitionary regime. It is worthy to note that although the problem is slow-fast on either side of the transitionary regime, Fenichel theory cannot be applied in a neighbourhood of the transition itself due to the loss of smoothness there as $\eps \to 0$. For this reason, we appeal to geometric blow-up techniques together with the more the general invariant manifold theory used in \cite{Nipp2013} (this is the formulation that we use, however see \cite{Hirsch1970,Shub2013} for earlier work and additional references).
	
	\
	
	The manuscript is structured as follows: In Section \ref{sec:setup_and_assumptions} we define and formulate the class of asymptotically slow-fast systems \eqref{eq:asymptotically_sf_systems}. This includes an analysis of the corresponding PWS system and the slow-fast systems corresponding to the forward and backward limits in time. Section \ref{sec:main_results} is dedicated to the statement and explanation of the main results. The geometric blow-up, which may be considered as both part of the mathematical formalism and part of the proof of the main results, is presented in Section \ref{sec:geometric_blow-up}. Section \ref{sec:proofs} is devoted to the proofs of the main results: Results on bifurcation and rate induced transitions are proven in Section \ref{sec:proof_of_thm_tipping}, and results on tracking and persistent normally hyperbolic manifolds are proven in Section \ref{sec:proof_of_thm_normal_hyperbolicity}. Our results are applied to low dimensional examples in Section \ref{sec:example}. Finally in Section \ref{sec:summary_and_outlook}, we summarise our discuss our findings.

	\section{Setup and assumptions}
	\label{sec:setup_and_assumptions}
	
	We consider ODE systems of the form
	\begin{equation}
		\label{eq:main}
		x' = \frac{dx}{dt} = f \left(x, \gamma(\mu t), \eps t, \sigma, \eps \right) ,
	\end{equation}
	where $x \in \mathbb R^n$, $f:\mathbb R^n \times \mathbb R \times \R \times \Lambda \times \rev{[0,\eps_0)} \to \mathbb R^n$ is $C^k$-smooth ($k$ will be assumed to be sufficiently large throughout), $\mu \in \Theta \subset \R$ is a positive `rate parameter', $\sigma \in \Lambda \subset \R$ is a bifurcation parameter, $\Lambda$ and $\Theta$ are bounded open intervals and \rev{$\eps \in [0,\eps_0)$} is a small time-scale parameter. Generalisations to the multi-parameter case with $\sigma \in \R^p$ where $p > 1$ are straightforward but omitted for simplicity. The function $\gamma:\mathbb R \to \mathbb R$ is a bounded and $C^k$ function which satisfies
	\begin{equation}
		\label{eq:gamma_limits}
		\gamma(z) \to 
		\begin{cases}
			1^- & \text{as } z \to \infty , \\
			0^+ & \text{as } z \to -\infty ,
		\end{cases}
	\end{equation}
	with
	\begin{equation}
		\label{eq:gamma_asymptotics}
		\gamma(-z^{-1}) = O(z^l) , \qquad 
		\gamma(z^{-1}) = 1 + O(z^l) ,
	\end{equation}
	as $z \to 0^+$ for some positive integer $l$. We take a dynamical systems approach and rewrite \eqref{eq:main} as an autonomous system in $n + 1$ variables by defining
	\begin{equation}
		\label{eq:s}
		s := \eps t , \qquad 0 < \eps \ll 1 ,
	\end{equation}
	and considering the system
	\begin{equation}
		\label{eq:main_lifted}
		\begin{split}
			x' &= 
			f\left( x, \gamma\left(\mu s / \eps \right) , s, \sigma, \eps \right) , \\
			s' &= 
			\eps ,
		\end{split}
	\end{equation}
	\rev{with $\eps \in (0, \eps_0)$.} The key observation at this point is that system \eqref{eq:main_lifted} is a slow-fast system with a PWS singular limit. Specifically, taking $\eps \to 0$ leads to
	\begin{equation}
		\label{eq:main_pws}
		\begin{pmatrix}
			x' \\
			s' 
		\end{pmatrix}
		=
		\begin{cases}
			\begin{pmatrix}
				f_-(x,s,\sigma) \\
				0 
			\end{pmatrix}
			& \quad \text{for} \quad s < 0 , \\
			\begin{pmatrix}
				f_+(x,s, \sigma) \\
				0
			\end{pmatrix}
			& \quad \text{for} \quad s > 0 , 
		\end{cases}
	\end{equation}
	where we defined $f_-(x,s,\sigma) := f(x,0,s,\sigma,0)$ and $f_+(x,s,\sigma) := f(x,1,s,\sigma,0)$. For a generic right-hand side, system \eqref{eq:main_pws} is discontinuous along a codimension-1 ($n$-dimensional) \textit{switching manifold}
	\[
	\Sigma := \left\{ (x,0) : x \in \mathbb R^n \right\} \subset \mathbb R^{n+1} .
	\]
	Notice that $s \in \R$ is a parameter for the limiting PWS system \eqref{eq:main_pws}. Generically, if the 1-parameter families of vector fields defined by $x' = f_\pm(x,s,\sigma)$ have isolated equilibria in $\R^n$, then system \eqref{eq:main_pws} can be expected to posses a union of 1-dimensional \textit{critical manifolds} in $\R^{n+1}$. We shall introduce a separate notation for critical manifolds on $\{s < 0\}$ and $\{s > 0\}$, namely
	\begin{equation}
		\label{eq:critical_manifolds}
		\begin{split}
			\rev{\mathcal C^-} &:= \left\{ (x,s) \in \R^n \times \R_- : f_-(x,s,\sigma) = 0 \right\} , \\
			\rev{\mathcal C^+} &:= \left\{ (x,s) \in \R^n \times \R_+ : f_+(x,s,\sigma) = 0 \right\} ,
		\end{split}
	\end{equation}
	respectively, where $\R_+ := [0,\infty)$ and $\R_- := (-\infty,0]$ (note that we include their extension up to $s = 0$). An example of one possible limiting situation is sketched in Figure \ref{fig:PWS}.
	
	\begin{remark}
		The asymptotics in \eqref{eq:gamma_asymptotics} imply that
		\begin{equation}
			\label{eq:f_pm_asypmtotics}
			x' = f(x, \gamma(\mu s/\eps), s, \sigma, \eps) = 
			\begin{cases}
				f_-(x,s,\sigma) + O(\eps) & \text{as } \eps \to 0 \text{ on } \{s < 0\} , \\
				f_+(x,s,\sigma) + O(\eps) & \text{as } \eps \to 0 \text{ on } \{s > 0\} ,
			\end{cases}
		\end{equation}
		which leads to
		\[
		\begin{pmatrix}
			x' \\
			s'
		\end{pmatrix} 
		\sim 
		\begin{cases}
			\begin{pmatrix}
				f_-(x,s,\sigma) + \rev{O(\eps)} , \\
				\eps
			\end{pmatrix} , & s < 0 , \\
			\begin{pmatrix}
				f_+(x,s,\sigma) + \rev{O(\eps)} , \\
				\eps
			\end{pmatrix} , & s > 0 , 
		\end{cases}
		\]
		as $\eps \to 0$. This representation highlights the fact that the asymptotic dynamics on either side of $\Sigma$ are slow-fast, which motivates the \textit{asymptotically slow-fast} terminology.
	\end{remark}
	
	The preceding observations show that the limit as $\eps \to 0$ is singular in two ways: Firstly, the equations on $\R^{n+1} \setminus \Sigma$ are slow-fast. Secondly, the non-smooth limit indicates that an abrupt transition occurs as solutions traverse an $\eps$-dependent neighbourhood of $\Sigma$. This transition is condensed into $\Sigma$ in the limit $\eps \to 0$. Although it appears as a loss of smoothness along $\Sigma$ when $\eps \to 0$, it is worthy to note that this highly singular feature is actually a natural consequence of the time-scale separation which is inherent in equations \eqref{eq:main} and \eqref{eq:main_lifted}. Specifically, it follows from the assumption that the transition occurs on a time-scale $t$ which is fast relative to the variation which occurs with time-scale $\eps t$ (cf.~the second and third arguments of $f$ in equation \eqref{eq:main}). This leads to three distinct regimes:
	\begin{itemize}
		\item A \textit{left-outer regime} bounded within $\{ s < 0 \}$;
		\item An \textit{inner/transitionary regime} $\{ s = O(\eps) \}$;
		\item A \textit{right-outer regime} bounded within $\{ s > 0 \}$.
	\end{itemize}
	Since $s' > 0$ for all $\eps > 0$, solutions with initial conditions in the left-outer regime traverse the inner transitionary regime before passing into the right-outer regime.
	
	\begin{remark}
		\label{rem:S_dimensions}
		Generically, the critical manifolds \rev{$\mathcal C^\pm$} are 1-dimensional. Higher dimensional critical manifolds can be expected in applications if the vector fields $f_\pm$ are themselves slow-fast. For example, if $x = (x_1, x_2) \in \R^{n_1} \times \R^{n_2}$ and
		\[
		f_-(x,s,\sigma) = 
		\begin{pmatrix}
			h(x_1,x_2,s,\sigma) \\
			0
		\end{pmatrix}
		\]
		for some smooth function $h : \R^{n_1} \times \R^{n_2} \times \R \times \Lambda \to \R^{n_1}$, then $\rev{\mathcal C^-}$ is defined by the condition $h(x_1, x_2, s, \sigma) = 0$ and is, therefore, $(n_2 + 1)$-dimensional. In this work we focus on the generic situation of 1-dimensional critical manifolds \rev{$\mathcal C^\pm$}, but we expect the methods and formalism to extend naturally to the higher dimensional setting.
	\end{remark}
	
	
	\begin{figure}[t!]
		\centering
		\includegraphics[scale=0.35]{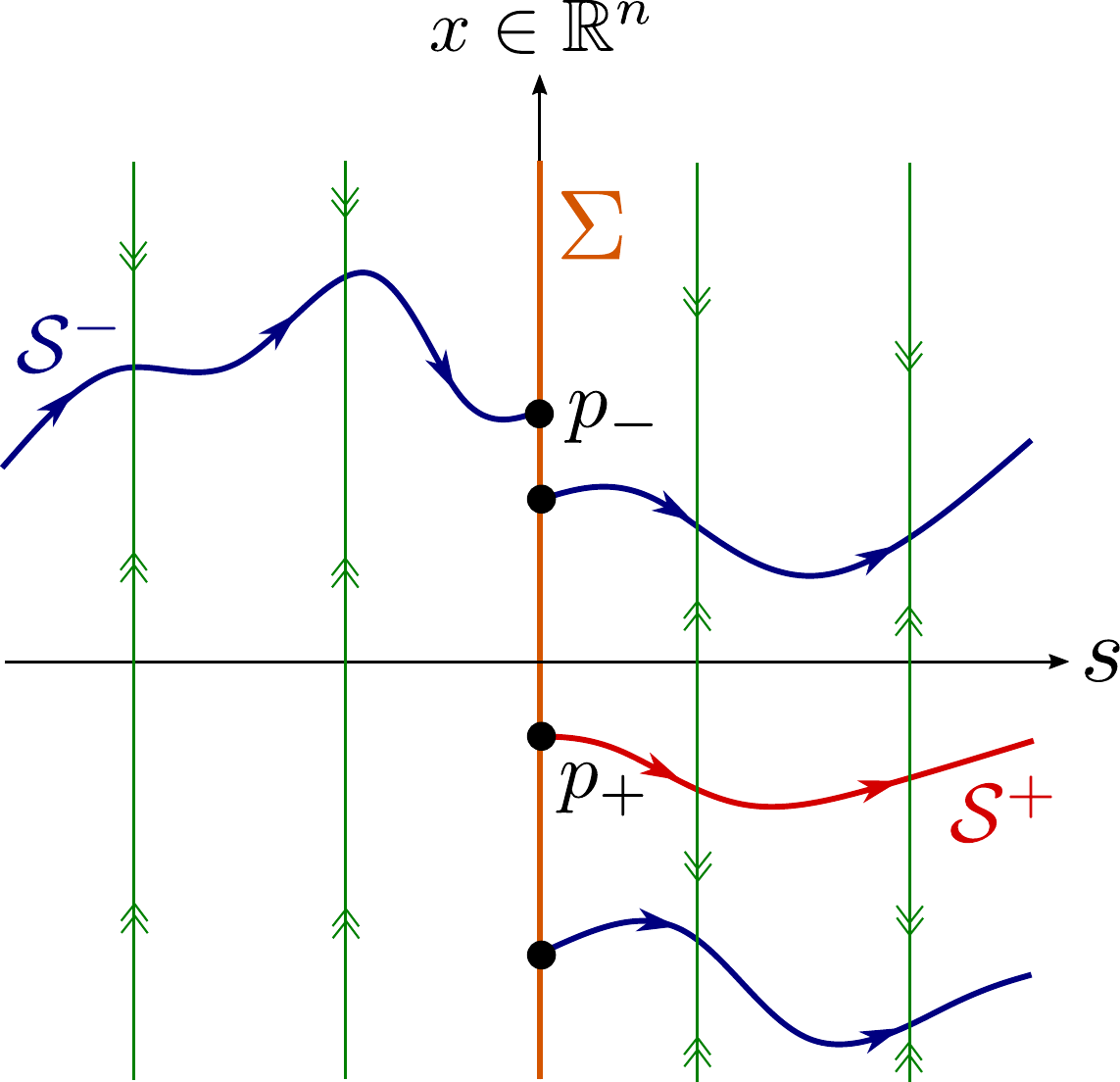}
		\caption{Possible geometry and dynamics for the limiting problem \eqref{eq:main_pws}, which is PWS with switching manifold $\Sigma = \{ s = 0 \}$ (shown in dark orange). We sketch the layer and reduced dynamics in both the left- and right-outer regimes $\{ s < 0 \}$ and $\{ s > 0 \}$ respectively. Fast fibers are parallel to the $x$-axis/axes, and shown in green with double arrows to indicate fast/hyperbolic flow. We sketch the case in which $\rev{\mathcal C^-}$ consists of a single attracting branch, and $\rev{\mathcal C^+}$ consists of 2 attracting branches and one repelling branch. Attracting (repelling) branches are sketched in blue (red). The reduced flow on both \rev{$\mathcal C^\pm$} is governed by $\dot s = 1$, and indicated with single arrows. Only the branches $\mathcal S^\pm \subseteq \rev{\mathcal C^\pm}$ that are subject to Assumption \ref{ass:normal_hyperbolicity} and \rev{their} intersections $p_\pm$ with $\Sigma$ are notated. Similar conventions on the arrows and colouring will be used throughout.}
		\label{fig:PWS}
	\end{figure}

	We now introduce an assumption on the existence and stability of the critical manifolds \rev{$\mathcal C^\pm$} close to $\Sigma$. It can be viewed as a kind of regularity or genericity condition.
	
	\begin{assumption}
		\label{ass:normal_hyperbolicity}
		The intersections $\rev{\mathcal C^\pm} \cap \Sigma$ are non-empty for all $\sigma \in \Lambda$. In particular, we have
		\[
		p_- : (x_-,0) \in \rev{\mathcal C^-} \cap \Sigma , \qquad
		p_+ : (x_+,0) \in \rev{\mathcal C^+} \cap \Sigma ,
		\]
		where $x_\pm \in \R^n$. Moreover, the eigenvalues of the $n \times n$ matrices $\Diff_x f_-(p_-)$ and $\Diff_x f_+(p_+)$ are uniformly bounded away from the imaginary axis for all $\sigma \in \Lambda$.
	\end{assumption}
	
	\begin{remark}
		For many problems it is sufficient to verify Assumption \ref{ass:normal_hyperbolicity} for a specific $\sigma$ and subsequently choose $\Lambda$ small enough to ensure that it holds for all $\sigma \in \Lambda$.
	\end{remark}
	
	Given Assumption \ref{ass:normal_hyperbolicity}, a direct application of the implicit function theorem shows that $\rev{\mathcal C^-}$ ($\rev{\mathcal C^+}$) contains a branch $\mathcal S^-$ ($\mathcal S^+$) which intersects $\Sigma$ transversally as $s \to 0^-$ ($s \to 0^+$). In particular, there exists some $\rho > 0$ and two smooth (parameter-dependent) functions $h^- : [-\rho, 0] \times \Lambda \to \R^n$ and $h^+ : [0,\rho] \times \Lambda \to \R^n$ such that
	\begin{equation}
		\label{eq:mathcal_S}
		\mathcal S^- = \left\{ (h^-(s, \sigma), s) : s \in [-\rho, 0] \right\} , \qquad
		\mathcal S^+ = \left\{ (h^+(s, \sigma), s) : s \in [0, \rho] \right\} .
	\end{equation}
	For $\rho > 0$ sufficiently small, Assumption \ref{ass:normal_hyperbolicity} guarantees that $\mathcal S^- \setminus \Sigma$ ($\mathcal S^+ \setminus \Sigma$) is normally hyperbolic, with the same `stability type' for all $s \in (-\rho,0)$ ($s \in (0,\rho)$). The reduced flow on both $\mathcal S^\pm$ (and on all other branches of \rev{$\mathcal C^\pm$}) can be expressed in the $s$-coordinate chart; it has the simple form
	\[
	\dot s = 1, 
	\]
	where the dot denotes differentiation with respect to the slow time $\tau = \eps t = s$. Thus in particular, the reduced flow on $\mathcal S^-$ approaches $\Sigma$ from the left, while the reduced flow on $\mathcal S^+$ is oriented away from $\Sigma$.
	
	\begin{remark}
		Fenichel theory applies to compact normally hyperbolic submanifolds of \rev{$\mathcal C^\pm$} which are bounded away from $\Sigma$, but it cannot be applied in a neighbourhood of $\Sigma$ because of the loss of smoothness there as $\eps \to 0$.
	\end{remark}
	
	\begin{remark}
		If $\mathcal S^-$ ($\mathcal S^+$) are normally hyperbolic and attracting for all $s < 0$ ($s > 0$), then the nearby Fenichel slow manifold constitutes a pullback (pushforward) attractor of system \eqref{eq:main} if an additional `pinning condition' which uniquely determines the slow manifold is satisfied as $t \to - \infty$ ($t \to \infty$). Since we shall be primarily interested in the dynamics in a tubular neighbourhood of $\Sigma$, our analysis is local in time on the fast time-scale $t$, and we need not concern ourselves with the existence of pullback or pushforward attractors.
	\end{remark}
	
	Finally, we need an assumption which applies to the limiting system obtained in the inner regime, where $s = O(\eps)$. We return to the original equation \eqref{eq:main}, with \rev{$t = s / \eps$}, and consider \rev{the} limiting equation obtained by setting $\eps = 0$, namely
	\begin{equation}
		\label{eq:limit_problem}
		x' = f(x, \gamma(\mu t), 0, \sigma, 0) .
	\end{equation}
	System \eqref{eq:limit_problem} is asymptotically autonomous and, as we shall see, plays an important role in determining whether or not bifurcation or rate induced transitions occur. We will be interested in three different cases, each of which are distinguished in the following assumption.
	
	\begin{assumption}
		\label{ass:heteroclinic}
		One of the following assertions is true:
		\begin{enumerate}
			\item[(I)] The limiting system \eqref{eq:limit_problem} has a unique solution $\varphi(t)$ such that
			\begin{equation}
				\label{eq:varphi_limit}
				\lim_{t \to \pm \infty} \varphi(t) = x_\pm ,
			\end{equation}
			for all $\sigma \in \Lambda$ and $\mu \in \Theta$, where $x_\pm \in \R^n$ are the same two points appearing in Assumption \ref{ass:normal_hyperbolicity}.
			\item[(II)] Fix $\sigma \in \Lambda$. There exists a critical value $\mu_c \in \Theta$ such that the limiting problem \eqref{eq:limit_problem}$|_{\mu = \mu_c}$ has a unique solution $\varphi(t)$ satisfying \eqref{eq:varphi_limit}.
			\item[(III)] Fix $\mu \in \Theta$. There exists a critical value $\sigma_c \in \Lambda$ such that the limiting problem \eqref{eq:limit_problem}$|_{\sigma = \sigma_c}$ has a unique solution $\varphi(t)$ satisfying \eqref{eq:varphi_limit}.
		\end{enumerate}
	\end{assumption}
	
	The special solution $\varphi(t)$ in (I), (II) and (III) will play an important role in the identification of rate and bifurcation induced transitions in system \eqref{eq:main_lifted}.
	
	\begin{remark}
		The blow-up analysis in Section \ref{sec:geometric_blow-up} will show that verifying Assumption \ref{ass:heteroclinic} is equivalent to the identification of a heteroclinic connection between \rev{equilibria induced by the points} $p_-$ and $p_+$ in \rev{a compactification of the autonomous system obtained from} the limiting problem \eqref{eq:limit_problem} \rev{after introducing the additional dependent variable $s_2 = t$}. As is typical for global bifurcations, \rev{finding the relevant heteroclinic} is expected to be difficult in applications which do not exhibit additional structure, e.g.~symmetries or slow-fast structure, which allow for the identification of key global objects.
	\end{remark}

	\section{Main results}
	\label{sec:main_results}
	
	We are now in a position to state and describe our main results. We present (i) results on bifurcation and rate induced transitions, and (ii) sufficient conditions for normally hyperbolic connections from $s < 0$ to $s > 0$ (in which case no such transition can occur). We begin with the former.

	\subsection{Bifurcation and rate induced transitions}
	
	Our main results on bifurcation and rate induced transitions will be formulated as sufficient conditions for the intersection of distinguished locally invariant manifolds which perturb from $\mathcal S^\pm$ under variation in $\mu$ or $\sigma$. We begin with a simple existence statement for these manifolds, deferring the detailed statement, which is formulated in the blown-up space, to Section \ref{sec:geometric_blow-up}.
	
	\begin{lemma}
		\label{lem:slow_manifolds}
		Consider system \eqref{eq:main_lifted} under Assumption \ref{ass:normal_hyperbolicity}. There is an $\eps_0 > 0$ such that for all $\eps \in (0,\eps_0)$, there exists 1-dimensional locally invariant manifolds $\mathcal S^\pm_\eps$ such that $\mathcal S^-_\eps|_{s \in [-\rho,0)} \to \mathcal S^-$ and $\mathcal S^+_\eps|_{s \in (0, \rho]} \to \mathcal S^+$ in the $C^k$ metric as $\eps \to 0$.
	\end{lemma}
	
	The existence and properties of the manifolds $\mathcal S^\pm_\eps$ can be shown directly in the blown-up space using center manifold theory; we refer to Lemmas \ref{lem:K1_cm} and \ref{lem:K3_cm} for a more detailed formulation and proof. \rev{In the following, we say that a bifurcation (rate) induced transition occurs if there exists a critical value $\sigma = \sigma_c$ ($\mu = \mu_c$) such that the asymptotics of the forward extension of $\mathcal S^-_\eps$ into $\{ s > 0 \}$ differ significantly as $\sigma$ ($\mu$) is varied over $\sigma_c$ ($\mu_c$) shortly after the corresponding solution leaves neighbourhood of $\Sigma$ (i.e.~for small values $s \in (0, \rho]$).}
	
	\begin{remark}
		\rev{In this work we adopt the heuristic `definition' of bifurcation and rate induced transitions outlined above in order to avoid the additional technicalities involved in providing a precise mathematical definition. The reader may find it helpful to keep in mind the `prototypical situation' in which the forward extension of $\mathcal S_\eps^-$ into $\{s > 0\}$ `jumps' from one attracting slow manifold to another under variation in $\sigma$ or $\mu$. This situation is sketched in Figures \ref{fig:PWS}, \ref{fig:heteroclinic}, \ref{fig:blow-up} and occurs in the example considered in Section \ref{sec:example_tipping}.}
	\end{remark}
	
	\begin{remark}
		The manifolds $\mathcal S^\pm_\eps$ can also be obtained as extensions of Fenichel slow manifolds if $\mathcal S^\pm$ are normally hyperbolic in the outer regions bounded away from $\Sigma$, which is the expected scenario in applications. The additional generality is needed here because Assumption \ref{ass:normal_hyperbolicity} only provides a hyperbolicity condition at the points $p_\pm \in \Sigma$. This condition is sufficient to guarantee the spectral condition for normal hyperbolicity in the limiting PWS system \eqref{eq:main_pws} in suitable neighbourhoods about $p_\pm$, but Fenichel theory cannot be applied to conclude the existence of slow manifolds on these neighbourhoods due to the loss of smoothness along $\Sigma$ when $\eps = 0$. Nevertheless, we shall refer to $\mathcal S^\pm_\eps$ as slow manifolds \rev{throughout} in order to simplify the exposition. \rev{Note also that the forward and backward extensions of $\mathcal S^\pm$ may still be normally hyperbolic everywhere (including the transitionary regime), even though Fenichel theory does not apply; we return to this point in Section \ref{sec:proof_of_thm_normal_hyperbolicity}.}
	\end{remark} 
	
	Our aim is to provide sufficient conditions for the regular intersection of (forward and backward extensions of) $\mathcal S^+_\eps$ and $\mathcal S^-_\eps$ under variation in $\mu$ or $\sigma$. \rev{Here and throughout, we refer to an intersection as \textit{regular} if it occurs for a locally isolated value of $\sigma$ or $\mu$ (exactly which of the two parameters will be clear from context)}. As it turns out, the intersection conditions can be formulated in terms of the functions
	\begin{equation}
		\label{eq:mathcal_G_eps}
		\mathcal G_\eps(\mu, \sigma) := \int_{-\infty}^{\infty} \psi_1(t) \left[ t \frac{\partial f}{\partial s}(\varphi(t), \gamma(\mu t), 0, \sigma, 0) + \frac{\partial f}{\partial \eps} (\varphi(t), \gamma(\mu t), 0, \sigma, 0) \right] dt ,
	\end{equation}
	\begin{equation}
		\label{eq:mathcal_G_mu}
		\mathcal G_\mu(\sigma) := 
		\int_{-\infty}^\infty \psi_1(t) t \gamma'(\mu_c t) \frac{\partial f}{\partial \gamma} (\varphi(t), \gamma(\mu_c t),0,\sigma,0) dt ,
	\end{equation}
	and
	\begin{equation}
		\label{eq:mathcal_G_sigma}
		\mathcal G_\sigma(\mu) := 
		\int_{-\infty}^\infty \psi_1(t) \frac{\partial f}{\partial \sigma} (\varphi(t), \gamma(\mu t),0,\sigma_c,0) dt ,
	\end{equation}
	\rev{assuming the existence of $\psi_1(t) \in \R^n$,} an exponentially decaying solution to the \textit{adjoint variational equation}
	\begin{equation}
		\label{eq:adjoint_variational_eqn}
		\psi_1'(t) = - (D_x f)^\transpose\big|_{(\varphi(t),\gamma(\mu t),0,\sigma,0)} \psi_1(t) ,
	\end{equation}
	evaluated at $\mu = \mu_c$ in the case of \eqref{eq:mathcal_G_mu} and $\sigma = \sigma_c$ in the case of \eqref{eq:mathcal_G_sigma}. \rev{The e}xistence of an exponentially decaying solution $\psi_1(t)$ \rev{follows from} the hyperbolicity properties at $x_\pm$ (the limit points as $\varphi(t) \to \pm \infty$); recall Assumption \ref{ass:normal_hyperbolicity}\rev{, if the Jacobian matrices $D_x f_-(p_-)$ ($D_x f_+ (p_+)$) have at least one eigenvalue which has negative (positive) real part bounded away from zero}.\footnote{\rev{The author is grateful to an anonymous reviewer for pointing out that the existence of an exponentially decaying solution $\psi_1(t)$ does not follow from Assumptions \ref{ass:normal_hyperbolicity}-\ref{ass:heteroclinic} alone.}} The quantities in \eqref{eq:mathcal_G_eps}, \eqref{eq:mathcal_G_mu} and \eqref{eq:mathcal_G_sigma} can be viewed as Melnikov integrals, and they appear naturally in the proof of Theorem \ref{thm:tipping} below.
	
	We now state the main result \rev{of this section}.
	
	\begin{thm}
		\label{thm:tipping}
		Consider system \eqref{eq:main} under Assumption \ref{ass:normal_hyperbolicity}\rev{, and assume that the Jacobian matrices $D_x f_-(p_-)$ ($D_x f_+ (p_+)$) have at least one eigenvalue which has negative (positive) real part bounded away from zero. Then t}here exists an $\eps_0 > 0$ such that for all $\eps \in (0,\eps_0)$, we have the following:
		\begin{enumerate}
			\item (Rate-induced transition I) Fix $\sigma \in \Lambda$. If Assumption \ref{ass:heteroclinic}(I) holds and there exists an $\mu_c > 0$ such that $\mathcal G_\eps(\mu_c, \sigma) = 0$ and $( \partial \mathcal G_\eps / \partial \mu)(\mu_c, \sigma) \neq 0$, then there exists a smooth function $\tilde \mu_c : (0,\eps_0) \to \R$ such that the forward and backward extension of $\mathcal S_\eps^-$ and $\mathcal S_\eps^+$ regularly intersect when $\mu = \tilde \mu_c(\eps)$. In particular,
			\[
			\tilde \mu_c(\eps) = \mu_c + O(\eps) 
			\]
			as $\eps \to 0$.
			\item (Rate-induced transition II) Fix $\sigma \in \Lambda$. If Assumption \ref{ass:heteroclinic}(II) holds and $\mathcal G_\mu(\sigma) \neq 0$, 
			then there exists a function $\hat \mu_c : (0,\eps_0) \to \R$ such that that the forward and backward extension of $\mathcal S_\eps^-$ and $\mathcal S_\eps^+$ regularly intersect when $\mu = \hat \mu_c(\eps)$. Moreover,
			\begin{equation}
				\label{eq:mu_c_asymptotics}
				\hat \mu_c(\eps) = \mu_c - \frac{\mathcal G_\eps(\mu_c, \sigma)}{\mathcal G_\mu(\sigma)} \eps + O(\eps^2) 
			\end{equation}
			as $\eps \to 0$.
			\item (Bifurcation-induced transition I) Fix $\mu \in \Theta$. If Assumption \ref{ass:heteroclinic}(I) holds and there exists an $\sigma_c > 0$ such that $\mathcal G_\eps(\mu, \sigma_c) = 0$ and $( \partial \mathcal G_\eps / \partial \sigma) (\mu, \sigma_c) \neq 0$, then there exists a smooth function $\tilde \sigma_c : (0,\eps_0) \to \R$ such that the forward and backward extension of $\mathcal S_\eps^-$ and $\mathcal S_\eps^+$ regularly intersect when $\sigma = \tilde \sigma_c(\eps)$. In particular,
			\[
			\tilde \sigma_c(\eps) = \sigma_c + O(\eps) 
			\]
			as $\eps \to 0$.
			\item (Bifurcation-induced transition II) Fix $\mu \in \Theta$. If Assumption \ref{ass:heteroclinic}(III) holds and $\mathcal G_\sigma(\mu) \neq 0$, 
			then there exists a function $\hat \sigma_c : (0,\eps_0) \to \R$ such that that the forward and backward extension of $\mathcal S_\eps^-$ and $\mathcal S_\eps^+$ regularly intersect when $\sigma = \hat \sigma_c(\eps)$. Moreover,
			\begin{equation}
				\label{eq:sigma_c_asymptotics}
				\hat \sigma_c(\eps) = \sigma_c - \frac{\mathcal G_\eps(\mu, \sigma_c)}{\mathcal G_\sigma(\mu)} \eps + O(\eps^2) 
			\end{equation}
			as $\eps \to 0$.
		\end{enumerate}
	\end{thm}
	
	The proof of Theorem \ref{thm:tipping} is given in Section \ref{sec:proof_of_thm_tipping}, however it relies on a number of results which we derive using geometric blow-up in Section \ref{sec:geometric_blow-up}. The blow-up transformation allows us to resolve the loss of smoothness along $\Sigma$ as $\eps \to 0$. 
	The manifolds $\mathcal S^\pm_\eps$ appear as center-type manifolds in the blown-up space, which can be extended in forward and backward time into the inner regime $s = O(\eps)$, where the transition occurs. The regular intersection of these manifolds under variation of $\mu$ (in the case of Assertions 1-2) or $\sigma$ (in the case of Assertions 3-4) is proven using an extension of the Melnikov method which was developed for problems on unbounded domains in \cite{Wechselberger2002}. Figure \ref{fig:heteroclinic} depicts the situation described by either Assertion 2 or 4 (depending on whether we vary $\mu$ or $\sigma$) in the blown-up space, in a situation where the outer dynamics is similar to the outer dynamics identified on the PWS level in Figure \ref{fig:PWS}. 
	The figure shows what we shall see in detail in Sections \ref{sec:geometric_blow-up} and \ref{sec:proofs}, namely, that the intersections described in Assertions 2 and 4 of Theorem \ref{thm:tipping} correspond to a heteroclinic connection between two points $p_-$ and $p_+$ which lie on the `equator' of the cylinder which is obtained after blowing up the switching manifold $\Sigma$.\footnote{\label{foot:notation}\rev{In order to simplify the exposition here, we permit an abuse of notation by conflating the points $p_\pm$ in Figure \ref{fig:PWS} with those in Figures \ref{fig:heteroclinic} and \ref{fig:blow-up}. Strictly speaking, the points $p_\pm$ which are defined in Assumption \ref{ass:normal_hyperbolicity} and shown in Figure \ref{fig:PWS} are blown-up to circles under \eqref{eq:blow-up}. These circles contain the equilibria $p_\pm$ depicted in Figures \ref{fig:heteroclinic} and Figure \ref{fig:blow-up}. This distinction is clarified in later sections by denoting the points in Figures \ref{fig:heteroclinic} and \ref{fig:blow-up} by $P_\pm$.}} The situation described in Assertions 1 and 3 is similar, except that in these cases the separation between the extended slow manifolds $\mathcal S^-_\eps$ and $\mathcal S^+_\eps$ is higher order in $\eps$. In general, Assertions 1 and 3 are expected to be useful in situations where the transition is triggered by $\mu$ or $\sigma$ variations in the leading order perturbations at $O(\eps)$.
	
	\begin{figure}[t!]
		\centering
		\includegraphics[scale=0.16]{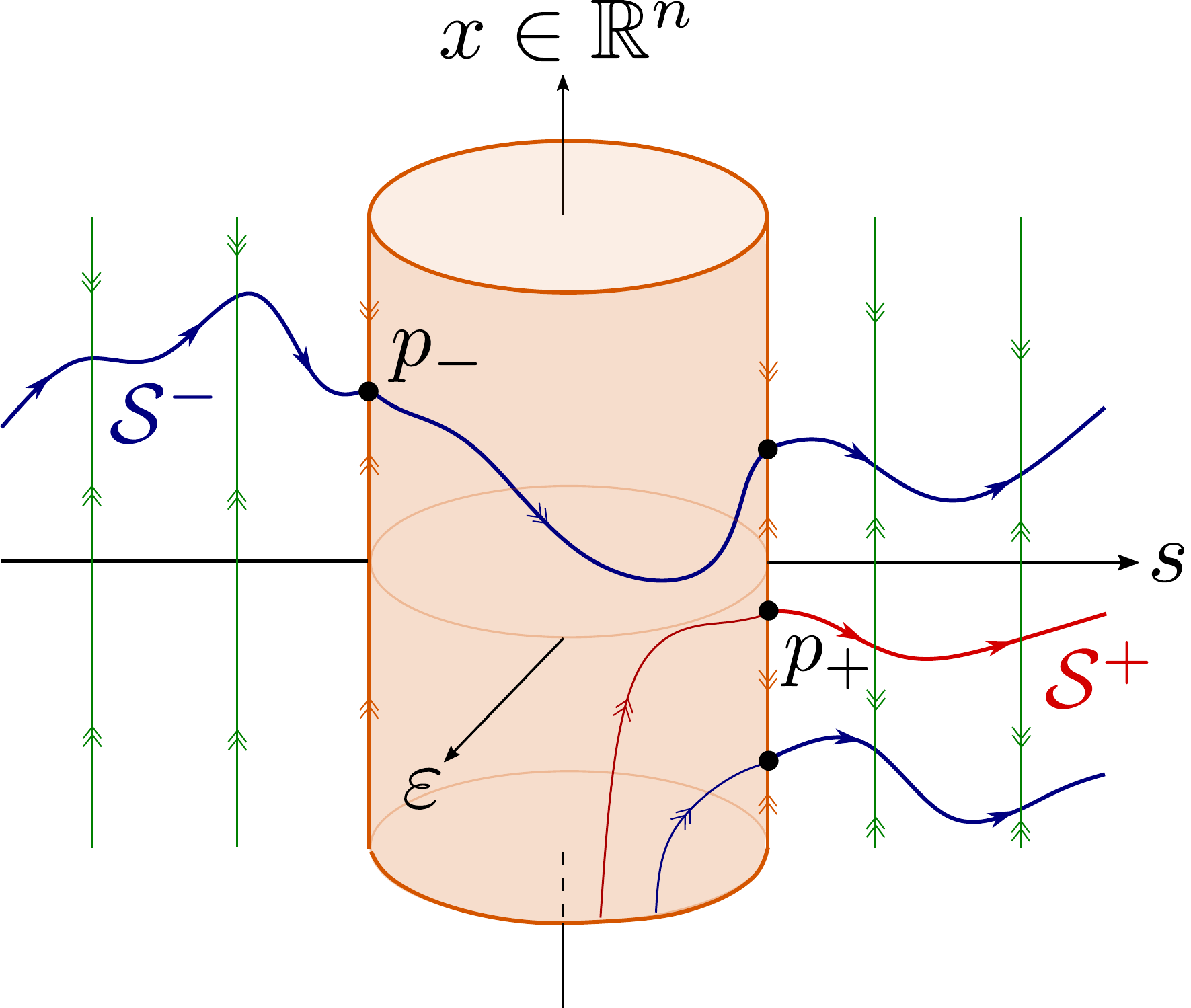}
		\ 
		\includegraphics[scale=0.16]{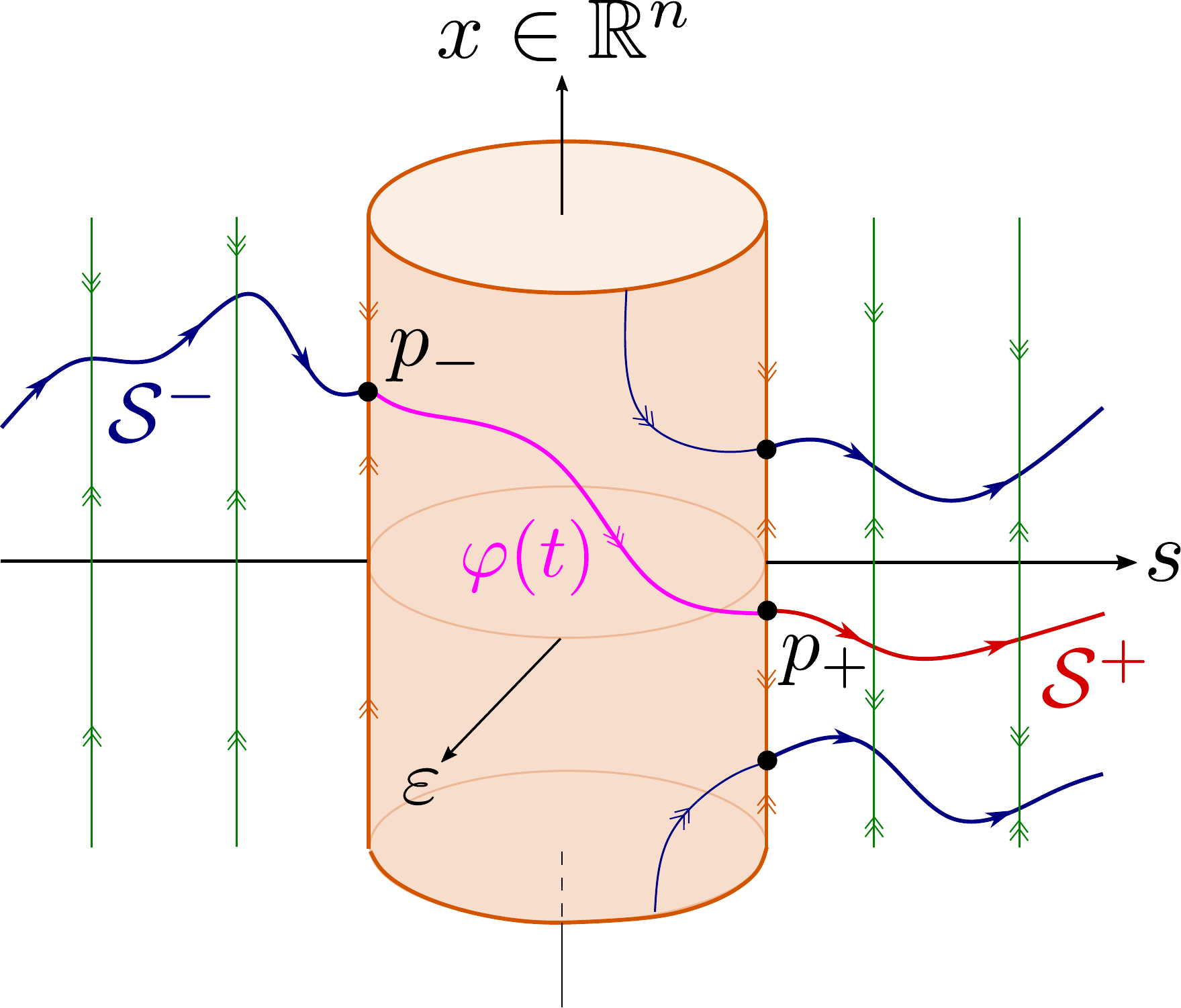}
		\ 
		\includegraphics[scale=0.16]{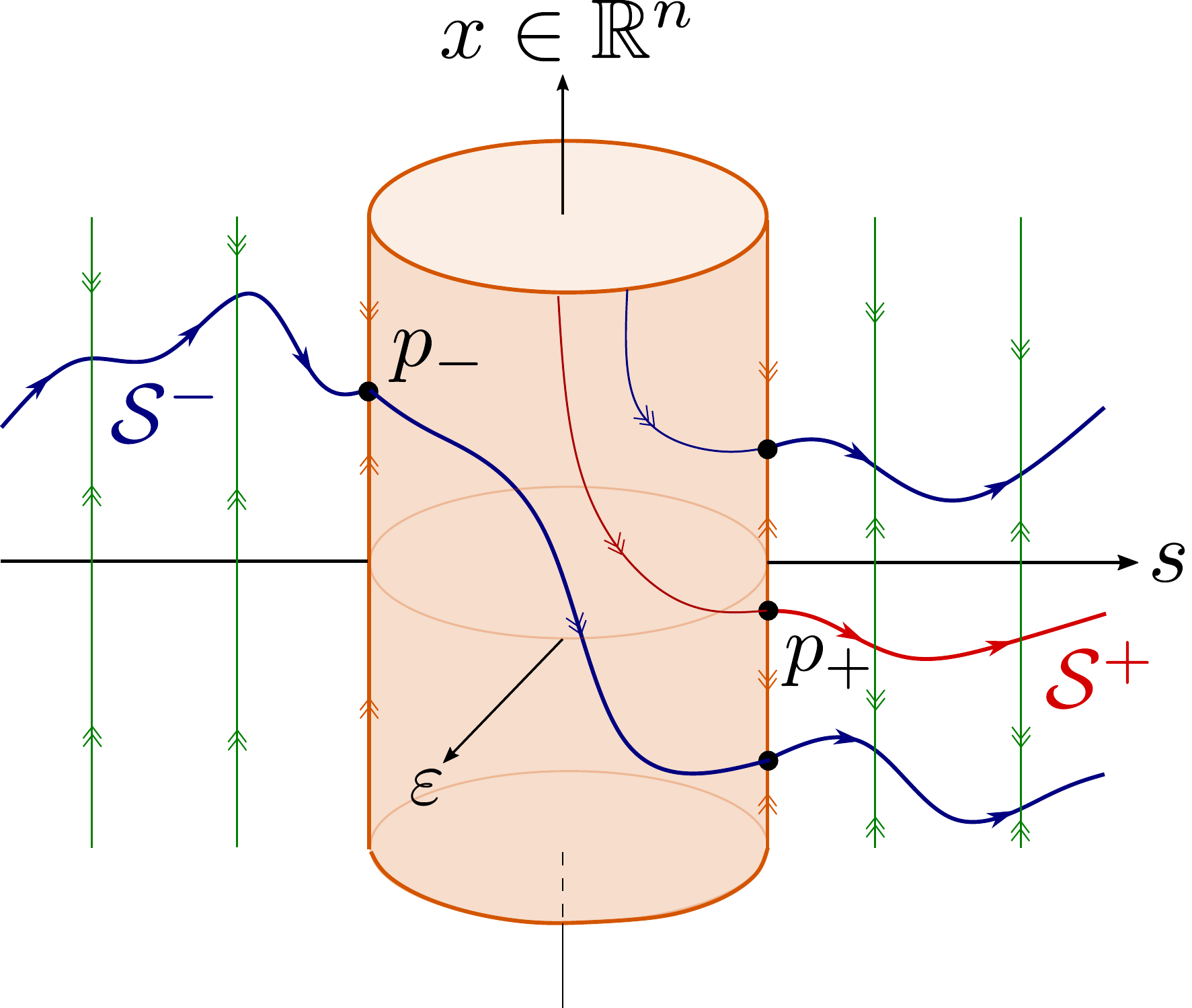}
		\caption{Schematic representation of the geometry and dynamics described by Theorem \ref{thm:tipping} Assertions 2 and 4 in the blown-up space, where the switching manifold $\Sigma$ has been `replaced' by the cylinder shown in shaded orange; see Section \ref{sec:geometric_blow-up} and in particular Figure \ref{fig:blow-up} and its caption for further details. If $\mathcal G_\mu(\sigma) < 0$ ($\mathcal G_\sigma(\mu) < 0$), then the left, center and right figures show the singular dynamics for $\mu < \mu_c$, $\mu = \mu_c$ and $\mu > \mu_c$ respectively ($\sigma < \sigma_c$, $\sigma = \sigma_c$ and $\sigma > \sigma_c$ respectively). The special solution $\varphi(t)$ in magenta occurs for $\mu = \mu_c$ or $\sigma = \sigma_c$, recall Assumption \ref{ass:heteroclinic}(II)-(III), is constrained to the blow-up cylinder, and forms a heteroclinic connection between two points $p_-$ and $p_+$ which connect to the critical manifolds $\mathcal S^-$ and $\mathcal S^+$ respectively \rev{(note that the points $p_\pm$ are closely related but not identical to the points $p_\pm$ in Figure \ref{fig:PWS}; see Footnote \ref{foot:notation})}. This heteroclinic connection breaks regularly under variation of $\mu$ if Theorem \ref{thm:tipping} Assertion 2 applies, or $\sigma$ if Theorem \ref{thm:tipping} Assertion 4 applies. Theorem \ref{thm:tipping} states that a corresponding canard-like intersection of the slow manifolds $\mathcal S^-_\eps$ and $\mathcal S^+_\eps$, which lie above the blow-up surface, occurs for all $\eps > 0$ sufficiently small.}
		\label{fig:heteroclinic}
	\end{figure}
	
	Theorem \ref{thm:tipping} applies in arbitrary dimensions, but in general one can only expect that the intersections described by Theorem \ref{thm:tipping} correspond to rate and bifurcation induced transitions in the case of scalar equations \eqref{eq:main}. In this case, the intersections are codimension-1, and therefore form a separatrix which breaks under variation in $\mu$ (in the case of Assertions 1-2) or $\sigma$ (in the case of Assertions 3-4). From a practical point of view, Theorem \ref{thm:tipping} also provides a means of \textit{determining} the critical rate $\mu_c$ or parameter $\sigma_c$, at least asymptotically as $\eps \to 0$. Specifically, $\sigma_c$ and $\mu_c$ can \rev{(in the case of Assertions 1 and 3) be determined by finding the roots of the equation $\mathcal G_\eps(\mu,\sigma) = 0$, or} (in the case of Assertions 2 and 4) by showing that Assumption \ref{ass:heteroclinic} is satisfied, \rev{i.e.~by verifying the existence of the special solution $\varphi(t)$ to the limiting problem \eqref{eq:limit_problem}}. The fact that the critical rates can be calculated is significant, given that the quantitative approximation of critical parameters in general non-autonomous systems i.e.~beyond the class of asymptotically autonomous systems \eqref{eq:asymptotically_autonomous_system}, tends to be \rev{a very difficult problem in practice}. 
	
	\begin{remark}
		The intersections which mediate the rate and bifurcation induced transitions described by Theorem \ref{thm:tipping} (when applied to scalar systems \eqref{eq:main}) imply the existence of a novel kind of \textit{canard solution} for asymptotically slow-fast systems, since it corresponds to the intersection of attracting and repelling slow manifolds which have been extended in forward and backward time respectively. We refer again to \cite{DeMaesschalck2021,Dumortier1996,Krupa2001a,Krupa2001b,Szmolyan2001,Wechselberger2012} for background and geometric analyses of canards in classical slow-fast systems which share a lot of qualitative similarities with the canard solutions described by Theorem \ref{thm:tipping}. \rev{The similarity to canards also leads naturally to the question of how one should define a corresponding way-in/way-out function. Since we are only interested in the identification of canard-like intersections in this work, this is left for future work.}
	\end{remark}
	
	\begin{remark}
		In principle, the tools that we use to prove Theorem \ref{thm:tipping}, e.g.~geometric blow-up, center manifold theory, perturbation theory and Melnikov theory for non-compact domains, also apply in arbitrary dimensions. We therefore expect that similar methods can be used to \rev{generalise Theorem \ref{thm:tipping} so that it can be used to identify rate and bifurcation induced transitions in higher dimensions}. The key requirement is that the slow manifolds $\mathcal S_\eps^\pm$ are codimension-1, i.e.~$n$-dimensional. This may occur in higher dimensions if, for example, the fast and future systems are slow-fast in the $x$ variables with only one fast direction; we refer again to Remark \ref{rem:S_dimensions}.
	\end{remark}

	\subsection{Sufficient conditions for normally hyperbolic tracking}
	
	In addition to the results on bifurcation and rate induced transitions described in Theorem \ref{thm:tipping}, we also present sufficient conditions 
	under which a normally hyperbolic and attracting slow manifold may preserve its smoothness, invariance and uniform attractivity properties as it extends through a neighbourhood of $\Sigma$. 
	\rev{Bifurcation and} rate induced transitions \rev{of the kind described by Theorem \ref{thm:tipping} are not possible in this case. We shall refer to this situation as \textit{normally hyperbolic tracking}, in order to distinguish it from notions of bifurcation or rate induced tracking which involve a qualitative change under the variation of a bifurcation or rate parameter (see e.g.~\cite{Duenas2023c}).} Since we are no longer interested in \rev{transitions}, we fix the system parameters $\mu$ and $\sigma$ write system \eqref{eq:main_lifted} more succinctly as
	\begin{equation}
		\label{eq:main_no_pars}
		\begin{split}
			x' &= f(x, \gamma(\mu s / \eps), s, \eps) , \\
			s' &= \eps .
		\end{split}
	\end{equation}
	
	
	%
	
	Let $\mathcal X \in \R^n$ denote an open, bounded and convex set such that the set $\mathcal X \times [-\rho, \rho]$ contains 
	the two critical manifolds $\mathcal S^\pm$. 
	We shall also write $\rev{\mathcal A} := \mathcal X \times [-\rho, \rho] \times (0,\eps_0)$. Our main result for this section will apply to systems of the form \eqref{eq:main_no_pars} which \rev{satisfy} the following assumption instead of Assumption \ref{ass:heteroclinic}.
	
	\begin{assumption}
		\label{ass:tracking}
		The set $\mathcal X$ and constants $\rho, \eps_0 > 0$ can be chosen such that the following properties hold:
		\begin{enumerate}
			\item[(i)] System \eqref{eq:main_no_pars} is inflowing with respect to $\mathcal X$, i.e.~the boundary $\partial \mathcal X$ is piecewise $C^1$ with an outward pointing normal $n(x)$ such that $n(x) \cdot f(x, \gamma(\mu s / \eps), s, \eps) < 0$ for all $(x,s,\eps) \in \rev{\mathcal A}$.
			\item[(ii)] There exists a constant $l_{22} < 0$ such the logarithmic norm inequality
			\begin{equation}
				\label{eq:log_norm}
				l_{22} := \sup_{(x,s,\eps) \in \rev{\mathcal A}} \left[ \lim_{h \to 0} \frac{\| \mathbb I_n + h D_x f \| - 1}{h} \right] < 0 ,
			\end{equation}
			holds, where $\mathbb I_n$ denotes the $n \times n$ identity matrix.
		\end{enumerate}
	\end{assumption}
	
	The properties in Assumption \ref{ass:tracking} have direct counterparts to the usual inflowing and normal contractivity assumptions which can be found in general normally \rev{hyperbolic} invariant manifold theory, see e.g.~\cite{Fenichel1971,Hirsch1970,Nipp2013,Wiggins2013}. In fact, both (i) and (ii) can be viewed as direct adaptions of some (but not all) of the basic assumptions that are used to prove the existence of attracting invariant manifolds in \cite{Nipp2013}; see in particular Hypotheses HD(a) and HD(c)(ii) therein. It is notable, however, that the list of assumptions in Assumption \ref{ass:tracking} is relatively short in comparison to the usual list of assumptions that one has to check in the (considerably more general) setting which is covered by these references.
	
	In order to state our main result, we let $I := [-\eps M, \eps M] \subset \R$, $\mathcal I := [-M, M] \subset \R$ and $s = \eps \hat s$, where $M \geq \rho^{-1}$ is an arbitrarily large but fixed constant, and define the constants
	\begin{equation}
		\label{eq:l_ij}
		l_{21} := \sup_{(x,s,\eps) \in \rev{\mathcal A}} \| D_x f \| , \qquad 
		l_{23} := \sup_{(x,s,\eps) \in \rev{\mathcal A}} \| D_\eps f \| ,
	\end{equation}
	where we have chosen a notation consistent with \cite{Nipp2013}, since the proof shall rely on results therein. \rev{Notice in particular that $\hat s = t$, due to \eqref{eq:s}.}
	
	We now state \rev{the} main result \rev{of this section}.
	
	\begin{thm}
		\label{thm:normal_hyperbolicity}
		Consider system \eqref{eq:main_no_pars} under Assumptions \ref{ass:normal_hyperbolicity} and \ref{ass:tracking}. There exists an $\eps_0 > 0$ and a $C^k$ function $\omega : [- \rho, \rho ] \times (0,\eps_0) \to \R^n$ such that the following assertions are true:
		\begin{enumerate}
			\item The 1-dimensional manifold defined by
			\begin{equation}
				\label{eq:mathcal_W}
				\mathcal W := \left\{ (\omega(s,\eps), s) : s \in [-\rho, \rho] \right\}
			\end{equation}
			is locally invariant under the flow of system \eqref{eq:main_no_pars}. Moreover, the restricted function $\omega_I := \omega|_{s \in I}$ satisfies the invariance equation
			\begin{equation}
				\label{eq:omega_I_invariance}
				\frac{\partial \omega_I}{\partial \hat s} (\hat s, \eps) = 
				f \left( \omega_I(\hat s, \eps), \gamma(\mu \hat s), \eps \hat s, \eps \right) ,
			\end{equation}
			for all $\hat s \in \mathcal I$ and $\eps \in (0, \eps_0)$.
			\item $\omega_I$ is Lipschitz in both $\hat s$ and $\eps$. More precisely, we have
			\[
			| \omega_I(\hat s, \eps) - \omega_I(\tilde{\hat s}, \eps) | \leq \frac{l_{21}}{|l_{22}|} , \qquad 
			| \omega_I(\hat s, \eps) - \omega_I(\hat s, \tilde \eps) | \leq \frac{l_{23}}{|l_{22}|} ,
			\]
			for all $\hat s, \tilde{\hat s} \in \mathcal I$ and $\eps, \tilde \eps \in (0,\eps_0)$.
			\item $\mathcal W$ is uniformly attracting, i.e.~all solutions $(x(t), s(t))$ with $(x(0), s(0)) \in \mathcal X \times (-\rho, \rho)$ satisfy
			\[
			\left| x(t) - \omega(s(t), \eps) \right| \leq 
			\e^{l_{22} t} \left| x(0) - \omega(s(0), \eps) \right| ,
			\]
			for all $t \geq 0$ such that $s(t) = \eps t + s(0) \in (-\rho, \rho)$.
			\item If $\mathcal S^\pm$ are normally hyperbolic and attracting for all $|s| \in (0,\rho]$, then for each fixed $c \in (0,\rho)$, $\mathcal W$ is $O(\eps)$ close to $\mathcal S^-$ ($\mathcal S^+$) on $s \in [-\rho, -c]$ ($s \in [c,\rho])$. More precisely, the restricted functions $\omega_O^- := \omega|_{s \in [-\rho,-c]}$ and $\omega_O^+ := \omega|_{s \in [c, \rho]}$ define slow manifolds with asymptotics
			\[
			\omega_O^\pm(s,\eps) = h^\pm(s) + O(\eps)
			\]
			as $\eps \to 0$\rev{, where $h^\pm$ are the (generally parameter-dependent) functions which appear in \eqref{eq:mathcal_S}}.
		\end{enumerate}
	\end{thm}
	
	Theorem \ref{thm:normal_hyperbolicity} is proven in Section \ref{sec:proof_of_thm_normal_hyperbolicity}, and uses Fenichel theory \cite{Fenichel1979,Jones1995,Kuehn2015,Wechselberger2019}, geometric blow-up, center manifold theory and the general invariant manifold theory developed in \cite{Nipp2013}. Roughly speaking, Fenichel theory covers the outer regimes and Assertion 4, while the inner regime and Assertions 1-3 can be treated using a standard cutoff procedure together with the results in \cite{Nipp2013}. The `boundaries' between the inner and outer regimes are described via center manifold theory after blow-up.
	
	\begin{remark}
		\label{rem:optimality}
		The choice of the neighbourhood $\mathcal X$ will in general effect the quantitative predictions made by Theorem \ref{thm:normal_hyperbolicity}, since this choice effects the optimality of the bounds $l_{21}$, $l_{22}$ and $l_{23}$ defined in equations \eqref{eq:log_norm} and \eqref{eq:l_ij}. We also note that the requirement that $\mathcal S^\pm \subset \mathcal X \times [-\rho, \rho]$ is made for simplicity; we are confident that a counterpart to Theorem \ref{thm:normal_hyperbolicity} could be proven in the case that $\mathcal X \times [-\rho, \rho]$ is replaced by a tubular neighbourhood which contains both $\mathcal S^\pm$ and satisfies the relevant inflowing conditions in Assumption \ref{ass:tracking}. This should be possible using standard techniques, namely (i) proving a local version of Theorem \ref{thm:normal_hyperbolicity}, and (ii) extending this to compact domains using a partition of unity.
	\end{remark}

	\section{Geometric blow-up}
	\label{sec:geometric_blow-up}
	
	In this section we present the geometric blow-up analysis. This will be necessary to prove Theorems \ref{thm:tipping} and \ref{thm:normal_hyperbolicity}, but it should also be viewed as an essential part of the general mathematical formalism for the analysis of asymptotically slow-fast systems \eqref{eq:asymptotically_sf_systems} developed herein. In essence, the blow-up will allow us to transform the problem in such a way that all three regimes (left-outer, inner, right-outer) can be viewed in a single `blown-up space' which has improved smoothness and hyperbolicity properties, albeit with a more complicated geometry.
	
	Recall that system \eqref{eq:main_lifted} suffers from a degeneracy in the sense that there is a loss of smoothness along $\Sigma$ as $\eps \to 0$. We shall resolve this using recent adaptations of the geometric blow-up method for regularised PWS systems; we refer again to \cite{Huzak2023,Jelbart2020d,Jelbart2021,Kristiansen2019c,Kristiansen2015a,Kristiansen2019d}. As is standard in blow-up approaches, we consider $\eps$ as a variable and work with the extended system obtained from \eqref{eq:main_lifted} after appending the trivial equation $\eps'=0$:
	\begin{equation}
		\label{eq:main_extended}
		\begin{split}
			x' &= f\left( x, \gamma\left(\mu s / \eps \right), s, \sigma, \eps \right) , \\
			s' &= \eps , \\
			\eps' &= 0 .
		\end{split}
	\end{equation}
	This system is degenerate along $\Sigma \times \{\eps = 0\} \subset \R^{n+2}$. The loss of smoothness will be resolved after applying the blow-up transformation
	\begin{equation}
		\label{eq:blow-up}
		r \geq 0 , \ \left(\bar s, \bar \eps \right) \in S^1 \mapsto 
		\begin{cases}
			s = r \bar s , \\
			\eps = r \bar \eps .
		\end{cases}
	\end{equation}
	Note that the variables $x \in \R^n$ are not effected by the transformation. Global coordinates in the `blown-up space' are given by $(x,\bar s,\bar \eps, r) \in \R^n \times S^1 \times \R_+$, and the degenerate set $\Sigma \times \{\eps = 0\}$ is replaced by the `cylinder' $\R^n \times S^1 \times \{0\}$; see Figure \ref{fig:blow-up}.
	
	\begin{figure}[t!]
		\centering
		\includegraphics[scale=0.25]{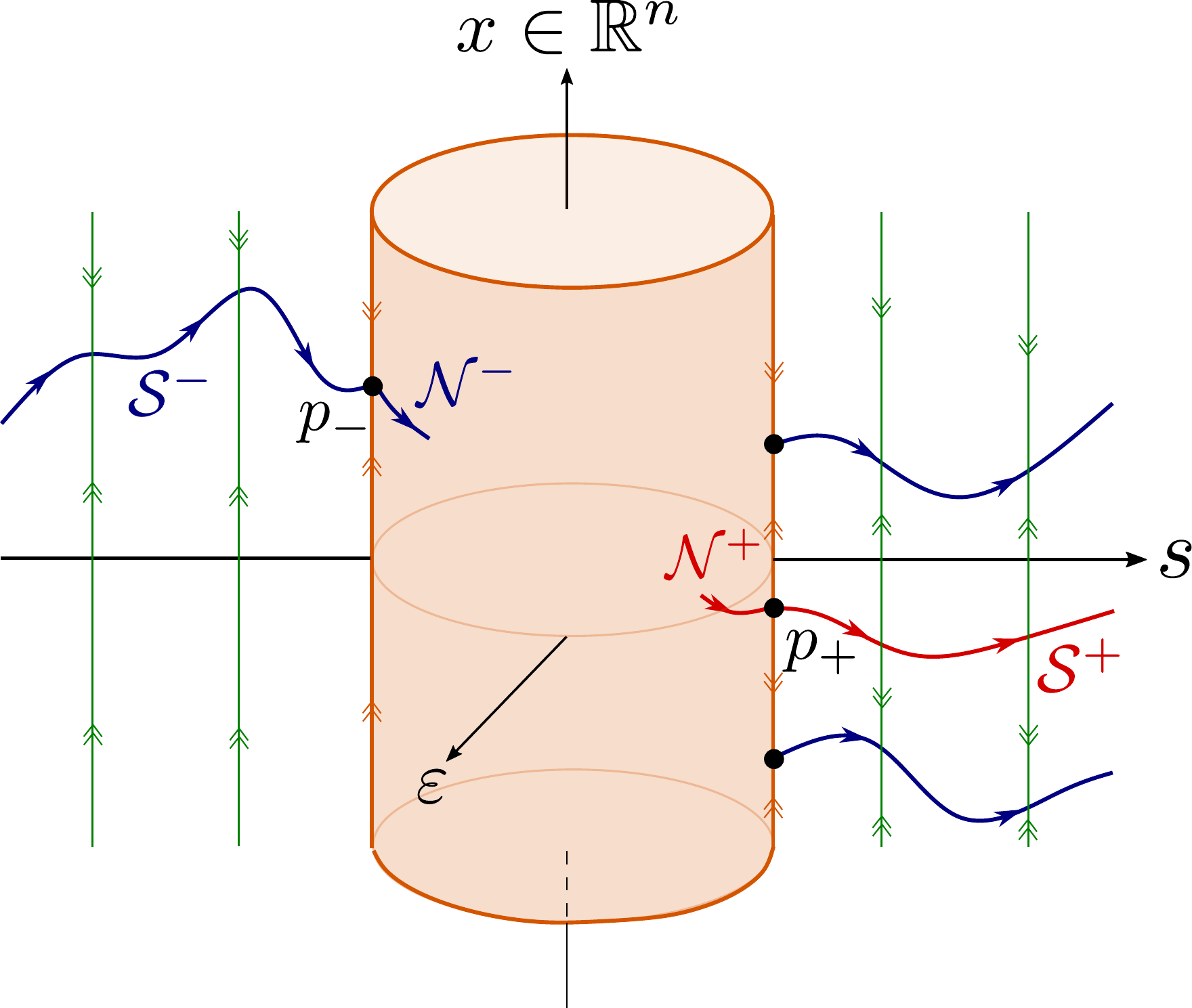}
		\caption{Possible geometry and dynamics in the blown-up space. We sketch the limiting dynamics for the associated left- and right-outer regimes (cf.~Figure \ref{fig:PWS}), which appear within the invariant regions defined by $\bar s = -1$, $\bar \eps = 0$, $r > 0$ and $\bar s = 1$, $\bar \eps = 0$ and $r > 0$ respectively. The switching manifold $\Sigma$ has been `blown-up' to the `cylinder' $\R^n \times S^1 \times \{r = 0\} \cong \R^n \times S^1$, which is shown in shaded orange. The limiting system \eqref{eq:limit_problem} in the inner transitionary regime induces dynamics within $r = 0$, i.e.~on the cylinder itself. In particular, solutions on the blow-up cylinder that are bounded in $x$ converge to the left equator with $\bar s = -1$, $\bar \eps = 0$ and $r = 0$ as $t \to -\infty$, and to the right equator with $\bar s = 1$, $\bar \eps = 0$ and $r = 0$ as $t \to \infty$. We also indicate the points $p_\pm$ \rev{(which we denote by $P_\pm$ in Sections \ref{sec:K1} and \ref{sec:K3} below; recall the comments in Footnote \ref{foot:notation})} and the local center manifolds $\mathcal N^\pm$ identified in the $K_1$ and $K_3$ analyses in Sections \ref{sec:K1} and \ref{sec:K3}.}
		\label{fig:blow-up}
	\end{figure}				
	
	For calculations, we shall work in three local coordinate charts, which can be defined via projective coordinates onto the (hyper)planes defined by
	\[
	K_1 : \ \bar s = -1 , \qquad K_2 : \ \bar \eps = 1 , \qquad K_3 : \ \bar s = 1 .
	\]
	Here and in what follows we adopt the (by now quite standard) notational conventions from \cite{Krupa2001a,Krupa2001c,Krupa2001b}. Local coordinates in $K_1$, $K_2$ and $K_3$ can be written as
	\begin{equation*}
		\begin{aligned}
			K_1 &: \ x = x , && s = - r_1 , && \eps = r_1 \eps_1 , \\
			K_2 &: \ x = x , && \rev{s = r_2 s_2} , && \eps = r_2 , \\
			K_3 &: \ x = x , && s = r_3 , && \eps = r_3 \eps_3 .
		\end{aligned}
	\end{equation*}
	Where their domains overlap, there is a smooth invertible change of coordinates between charts $K_1$ and $K_2$, and charts $K_2$ and $K_3$. These are given by
	\begin{equation}
		\label{eq:kappa_ij}
		\begin{aligned}
			\kappa_{12} &: \ r_1 = - r_2 s_2 , && \eps_1 = - s_2^{-1} , && s_2 < 0	 , \\
			\kappa_{23} &: \ r_2 = r_3 \eps_3 , && s_2 = \eps_3^{-1} , && \eps_3 > 0 .
		\end{aligned}
	\end{equation}
	Finally, we note that since $\eps$ is constant, the expressions $\eps = r_1 \eps_1$, $\eps = r_2$ and $\eps = r_3 \eps_3$ define invariant level sets in charts $K_1$, $K_2$ and $K_3$ respectively.
	
	\begin{remark}
		\label{rem:blow-up_regimes}
		The limiting problems associated with all three regimes, i.e.~the left-outer, inner and right-outer regimes, appear in different invariant regions of the blown-up space. The left-outer limiting problem corresponds to dynamics within the set defined by $\bar s = -1, \bar \eps = 0$ and $r > 0$, the inner limiting problem corresponds to dynamics on the blow-up cylinder itself, where $r = 0$, and the right-outer limiting problem corresponds to dynamics within the set defined by $\bar s = 1, \bar \eps = 0$ and $r > 0$. This is sketched in Figure \ref{fig:blow-up}.
	\end{remark}
	
	We turn now to the dynamics in the so-called \textit{entry} and \textit{exit} charts $K_1$ and $K_3$, respectively, deferring the details of the dynamics in the inner rescaling chart $K_2$ to Section \ref{sec:proofs}.
	
	
	

	\subsection{Dynamics in chart $K_1$}
	\label{sec:K1}
	
	Writing system \eqref{eq:main_extended} in local $K_1$ coordinates, we obtain
	\begin{equation}
		\label{eq:K1_eqns}
		\begin{split}
			x' &= f_-(x,-r_1,\sigma) + O(r_1 \eps_1, \eps_1^l) , \\
			r_1' &= -r_1 \eps_1 , \\
			\eps_1' &= \eps_1^2 ,
		\end{split}
	\end{equation}
	where 
	the equation for $x'$ was obtained via Taylor expansion using \eqref{eq:gamma_asymptotics}. We are primarily interested in the dynamics in a neighbourhood defined by $\eps_1 \in [0,\delta]$ and $r_1 \in [0,\rho]$, for arbitrarily small but fixed $\delta, \rho > 0$. 
	
	\begin{remark}
		Typically, blow-up analyses involve a non-trivial transformation of time which amounts to the formal division or multiplication of the right-hand side in \eqref{eq:K1_eqns} by a certain power of $r_1$ (division in the case of blow-up at non-hyperbolic points or manifolds, and multiplication in the case of blow-up at non-smooth points or manifolds). Usually, this is necessary in order to obtain a well-defined and non-trivial vector field on the blow-up surface, i.e.~for $r_1 = 0$. Interestingly, for asymptotically slow-fast systems of the form \eqref{eq:main_lifted}, the desingularisation step is not necessary in order to obtain the `already desingularised' equations \eqref{eq:K1_eqns}. The analogous feature also holds for the $K_2$ and $K_3$ analyses below.
	\end{remark}
	
	System \eqref{eq:K1_eqns} has a 1-dimensional critical manifold containing the branch
	\begin{equation}
		\label{eq:S1_minus}
		\mathcal S^-_1 := \left\{ (h_0^-(-r_1, \sigma),r_1,0) : r_1 \in [0,\rho] \right\} ,
	\end{equation}
	which corresponds to the blown-up preimage of $\mathcal S^-$ for $r_1 > 0$, now extended up to $r_1 = 0$, i.e.~up to the point $P_- : (x_-(\sigma),0,0)$ (\rev{which we denote by $p_-$ in Figures \ref{fig:heteroclinic} and \ref{fig:blow-up}), where we r}ecall that $h^-_0(0,\sigma) = x_-(\sigma)$. We are interested in the geometry and dynamics near $P_-$.
	
	\begin{lemma}
		\label{lem:K1_cm}
		Fix $\rho, \delta > 0$ sufficiently small. Then there exists a two-dimensional, $C^k$-smooth center manifold
		\[
		\mathcal M_1^- := \{ (h_1^-(r_1,\eps_1,\sigma), r_1, \eps_1) : r_1 \in [0,\rho], \eps_1 \in [0,\delta] \} , 
		\]
	where $h_1^-(r_1,\eps_1,\sigma) = x_-(\sigma) + O(r_1, \eps_1)$. 
	In particular,
	\[
	\mathcal M_1^- |_{\eps_1 = 0} = \mathcal S_1^- , \qquad
	\mathcal M_1^- |_{r_1 = 0} = \mathcal N_1^- ,
	\]
	where $\mathcal S_1^-$ is the critical manifold defined in \eqref{eq:S1_minus} and $\mathcal N_1^-$ is a 1-dimensional center manifold emanating from $P_-$ in $\{r_1 = 0\}$.
	\end{lemma}
	
	\begin{proof}
	Assumption \ref{ass:normal_hyperbolicity} implies that the Jacobian matrix appearing in the linearisation of system \eqref{eq:K1_eqns} at $P_-$ has $n$ eigenvalues that are bounded away from the imaginary axis, and 2 zero eigenvalues due to the equations for $r_1'$ and $\eps_1'$. From here the result follows from direct calculations and center manifold theory.
	\end{proof}
	
	\rev{Lemma \ref{lem:K1_cm} explains the local geometry near $P_-$ in Figure \ref{fig:blow-up}, which shows the relative positioning of the local critical and center manifolds $\mathcal S_1^-$ and $\mathcal N_1^-$ respectively, along which $\mathcal M_1^-$ has its base (they are depicted without the subscript and in global coordinates).} The locally invariant manifold $\mathcal S^-_\eps$ described in Lemma \ref{lem:slow_manifolds} is obtained after applying the blow-down transformation $(r_1, \eps_1) = (-s, -\eps / s)$ to $\mathcal M^-_1$ on $s \in [-\rho, 0)$. It is also worthy to note that in the attracting case where all of the eigenvalues of $D_xf_-(p_-)$ are bounded inside the left-half plane, the center manifold $\mathcal M_1^\pm$ is exponentially attracting with a contraction rate determined by the size of the spectral gap.
	
	
	\begin{remark}
	\label{rem:past_limit}
	The dynamics within the invariant set $\{ r_1 = \eps_1 = 0 \}$ are governed by the \textit{autonomous} $n$-dimensional ODE system
	\[
	x' = f_-(x,0,\sigma) ,
	\]
	for which $x_-$ is a hyperbolic equilibrium due to Assumption \ref{ass:normal_hyperbolicity}. This is precisely the past limiting system that one would obtain by setting $\eps = 0$ in the original equation \eqref{eq:main} and taking $t \to - \infty$ in the resulting asymptotically autonomous system.
	\end{remark}
	

	\subsection{Dynamics in chart $K_3$}
	\label{sec:K3}
	
	We now consider the dynamics in $K_3$. Since the analysis closely resembles the analysis in $K_1$, we present fewer details. Writing system \eqref{eq:main_extended} in $K_3$ coordinates, we obtain
	\begin{equation}
	\label{eq:K3_eqns}
	\begin{split}
		x' &= f_+(x, r_3, \sigma) + O(r_3 \eps_3, \eps_3^l) , \\
		r_3' &= r_3 \eps_3 , \\
		\eps_3' &= - \eps_3^2 .
	\end{split}
	\end{equation}
	Similarly to the $K_1$ analysis, we are primarily interested in local dynamics within a neighbourhood defined by $\eps_3 \in [0,\delta]$ and $r_3 \in [0, \rho]$.
	
	System \eqref{eq:K3_eqns} has a 1-dimensional critical manifold
	\begin{equation}
	\label{eq:S3_+}
	\mathcal S^+_3 := \left\{ (x,h_0^+(r_3, \sigma),0) : r_3 \in [0,\rho] \right\} ,
	\end{equation}
	which corresponds to the blown-up image of \rev{$\mathcal S^+$} for $r_3 > 0$, now extended up to $r_3 = 0$, i.e.~up to the point $P_+ : (x_+(\sigma),0,0)$ \rev{(denoted by $p_+$ in Figures \ref{fig:heteroclinic} and \ref{fig:blow-up})}.
	
	\begin{lemma}
	\label{lem:K3_cm}
	Fix $\rho, \delta > 0$ sufficiently small. Then there exists a two-dimensional, $C^k$-smooth center manifold
	\[
	\mathcal M_3^+ := \{ (h_3^+(r_3,\eps_3,\sigma), r_3, \eps_3) : r_3 \in [0,\rho], \eps_3 \in [0,\delta] \} ,
	\]
	where $h_3^+(r_3,\eps_3,\sigma) = x_+(\sigma) + O(r_3, \eps_3)$. 
	In particular,
	\[
	\mathcal M_3^+ |_{\eps_3 = 0} = \mathcal S_3^+ , \qquad
	\mathcal M_3^+ |_{r_3 = 0} = \mathcal N_3^+ ,
	\]
	where $\mathcal S_3^+$ is the critical manifold defined in \eqref{eq:S3_+} and $\mathcal N_3^+$ is a 1-dimensional center manifold emanating from $P_+$ in $\{r_3 = 0\}$.
	\end{lemma}
	
	\begin{proof}
	The proof is similar to the proof of Lemma \ref{lem:K1_cm}, so we omit the details.
	\end{proof}
	
	\rev{Lemma \ref{lem:K1_cm} explains the local geometry near $p_+$ in Figure \ref{fig:blow-up}.} The locally invariant manifold $\mathcal S^+_\eps$ described in Lemma \ref{lem:slow_manifolds} is obtained after applying the blow-down transformation $(r_3, \eps_3) = (s, \eps / s)$ to $\mathcal M^+_3$ on $s \in (0,\rho]$. Combining this with the corresponding result from the chart $K_1$ analysis suffices to prove Lemma \ref{lem:slow_manifolds}.
	
	
	\begin{remark}
	An analogue to Remark \ref{rem:past_limit} applies. Namely, the dynamics within $r_3 = \eps_3 = 0$ are invariant and governed by the $n$-dimensional ODE system
	\[
	x' = f_+(x,0,\sigma) ,
	\]
	for which $x_+$ is a hyperbolic equilibrium due to Assumption \ref{ass:normal_hyperbolicity}. This is precisely the future limiting system that one would obtain by setting $\eps = 0$ in the original equation \eqref{eq:main} and taking $t \to \infty$ in the resulting asymptotically autonomous system.
	\end{remark}
	

	\section{Proofs}
	\label{sec:proofs}
	
	Based on the geometric formalism and results derived in Section \ref{sec:geometric_blow-up}, we now proceed with the proofs for Theorems \ref{thm:tipping} and \ref{thm:normal_hyperbolicity}. We begin with the former.

	\subsection{Proof of Theorem \ref{thm:tipping}}
	\label{sec:proof_of_thm_tipping}
	
	In the preceding sections, and in particular in Lemmas \ref{lem:K1_cm} and \ref{lem:K3_cm}, we identified center manifolds about the points $P_\pm$ which we shall denote here by $\mathcal M^\pm$ (i.e.~we shall drop the subscripts). These center manifolds can be used to describe the extension of the slow manifolds $\mathcal S^\pm_\eps$ into the inner transitionary regime where $s = O(\eps)$, which corresponds to the region covered by chart $K_2$ in the blown-up space. The aim is now to consider the extension of $\mathcal M^-$ ($\mathcal M^+$) under the forward (backward) flow induced by system \eqref{eq:main_extended}, and to show that these manifolds intersect transversally if the conditions of Assumption \ref{ass:heteroclinic}(I), (II) or (III) are satisfied.
	
	We work primarily in chart $K_2$, where the equations are given by
	\begin{equation}
	\label{eq:K2_eqns}
	\begin{split}
		x' &= f(x, \gamma(\mu s_2), r_2 s_2, \sigma, r_2) , \\
		s_2' &= 1 .
	\end{split}
	\end{equation}
	\rev{On compact domains, t}his is a regular perturbation problem with rate parameter $\mu$, bifurcation parameter $\sigma$, and perturbation parameter $0 < r_2 \ll 1$. In fact, it is naturally identified with the original system \eqref{eq:main} using $r_2 = \eps$ and $s = r_2 s_2 = \eps t \implies s_2 = t$. By Assumption \ref{ass:heteroclinic}, the limiting system 
	\begin{equation}
	\label{eq:K2_limiting_system}
	\begin{split}
		x' &= \rev{f(x, \gamma( \mu s_2), 0, \sigma, 0) } , \\
		s_2' &= 1 ,
	\end{split}
	\end{equation}
	has a special solution with parameterisation $\phi(t) = (\varphi(t),t)$, where $\varphi(t) \to x_\pm$ as $t \to \pm \infty$, either for all $\mu$ and $\sigma$ in the case of Assumption \ref{ass:heteroclinic}(I), for $\mu = \mu_c$ and fixed $\sigma$ in the case of Assumption \ref{ass:heteroclinic}(II), or for $\sigma = \sigma_c$ and fixed $\mu$ in the case of Assumption \ref{ass:heteroclinic}(III). In the blown-up space, this corresponds to a heteroclinic connection between the points $P_-$ and $P_+$ identified in the chart $K_1$ and $K_3$ analyses respectively; this is shown in the center subfigure of Figure \ref{fig:heteroclinic}. In particular, $\varphi(t)$ has at most algebraic growth as $t \to \pm \infty$, since it has base along the center manifolds $\mathcal N^-_1$ close to $P_-$ and $\mathcal N_3^+$ close to $P_+$. This follows from Lemmas \ref{lem:K1_cm} and \ref{lem:K3_cm} and the fact that the limiting system \eqref{eq:K2_limiting_system} is defined for $r_2 = 0$, which corresponds to $r_1 = 0$ in $K_1$ and $r_3 = 0$ in $K_3$ due to the change of coordinates formulae in \eqref{eq:kappa_ij} (in the regions where the charts overlap). 
	
	\begin{remark}
	The limiting system \eqref{eq:K2_limiting_system} corresponds to dynamics on the blow-up cylinder, since $r_2 = 0$. This system \rev{can be obtained directly from the asymptotically autonomous limiting system \eqref{eq:limit_problem} by introducing $s_2 = t$ as a new dependent variable and extending the phase space}. As $t \to -\infty$ ($t \to \infty$), bounded solutions of system \eqref{eq:K2_limiting_system} approach the left (right) intersection of the blow-up cylinder with the (hyper)plane $\bar \eps = 0$. In fact, applying the blow-up map and subsequently restricting to $r = 0$ is equivalent to making a compactification of time, where the dynamics at infinity can be studied in charts $K_1$ and $K_3$. In addition to this, the blow-up also allows one to connect to the outer limiting problems, as described in Remark \ref{rem:blow-up_regimes}, which is not possible in approaches based on compactification alone.
	\end{remark}
	
	As shown in \cite{Wechselberger2002}, a distance function which measures the separation of the \rev{manifolds} $\mathcal M^-$ and $\mathcal M^+$ can be formulated in terms of solutions to the adjoint problem associated to the \rev{(linear)} variational equations along $\phi(t)$. Denote the right-hand side of \eqref{eq:K2_eqns} by $F(x,s_2,\mu,\sigma,r_2)$. Then the relevant adjoint variational problem is given by
	\[
	\psi'(t) = - (\Diff_{(x,s_2)} F)^\transpose \big|_{\rev{(\phi(t),\mu,\sigma,0)}} \psi(t) = -
	\begin{pmatrix}
	(\Diff_xf)^\transpose & \mathbb O_{n,1} \\
	\left( \frac{\partial f}{\partial s_2} \right)^\transpose & 0
	\end{pmatrix}
	\bigg|_{\rev{(\varphi(t), \gamma (\mu t), 0, \sigma, 0)}} \psi(t) .
	\]
	If we separate components by writing $\psi(t) = (\psi_1(t), \psi_2(t)) \in \R^n \times \R$, then $\psi_1(t)$ solves the $n$-dimensional adjoint variational problem in \eqref{eq:adjoint_variational_eqn}. We now define a codimension-1 (i.e.~$n$-dimensional) vertical section by $\Omega := \{ (x,0) : x \in J \}$, where $J \subseteq \R^n$ is a neighbourhood about the point $\varphi(0)$. Since the orbit corresponding to the special solution $\phi(t)$ intersects $\Omega$ transversally, $\mathcal M^-$ and $\mathcal M^+$ are guaranteed to intersect $\Omega$ transversally as long as $r_2$ is sufficiently small and, in the case of Assumption \ref{ass:heteroclinic}(II) or (III), if the quantity $\mu - \mu_c$ or  $\sigma - \sigma_c$ respectively is also sufficiently small. The argument from this point forward is similar for all four assertions in Theorem \ref{thm:tipping}. However, it is necessary to consider each case separately.

	\subsubsection{Proving Assertions 1 and 3}
	
	Consider the case in which Assumption \ref{ass:heteroclinic}(I) holds. If $r_2$ is sufficiently small then the distance between the the $x$-coordinates of the intersections $\mathcal M^- \cap \Omega$ and $\mathcal M^+ \cap \Omega$ is shown in \cite{Wechselberger2002} to be measured by the signed distance function
	\[
	\mathcal D(\mu, \sigma, r_2) = 
	\frac{\partial \mathcal D}{\partial r_2}(\mu,\sigma,0) r_2 + O \left( r_2^2 \right) ,
	\]
	where the coefficient function takes the form of a \textit{Melnikov integral}, in this case given by
	\[
	\frac{\partial \mathcal D}{\partial r_2}(\mu,\sigma,0) = 
	\int_{-\infty}^\infty \left\langle \psi(t), \frac{\partial F}{\partial r_2} (\phi(t),\mu,\sigma,0) \right\rangle dt =
	\int_{-\infty}^\infty \psi_1(t) \frac{\partial f}{\partial r_2} (\varphi(t),\gamma(\mu t),0,\sigma,0) dt .
	\]
	Using the chain rule and $\eps = r_2$ to evaluate the partial derivative in the integrand, we obtain
	\[
	\frac{\partial \mathcal D}{\partial r_2}(\mu,\sigma,0) = 
	\int_{-\infty}^{\infty} \psi_1(t) \left[ t \frac{\partial f}{\partial s}(\varphi(t), \gamma(\mu t), 0, \sigma, 0) + \frac{\partial f}{\partial \eps} (\varphi(t), \gamma(\mu t), 0, \sigma, 0) \right] dt
	= \mathcal G_\eps(\mu,\sigma) ,
	\]
	i.e.~precisely the function defined in \eqref{eq:mathcal_G_eps}, where $\psi_1(t)$ solves the adjoint variational equation \eqref{eq:adjoint_variational_eqn} and decays exponentially as $t \to \pm \infty$. \rev{The existence of such a solution follows from Assumptions \ref{ass:normal_hyperbolicity}, \ref{ass:heteroclinic} and the requirements on the spectrum of $D_x f_\pm(p_\pm)$; recall the discussion just prior to the statement of Theorem \ref{thm:tipping}}.
	
	Now define the rescaled distance function
	\[
	\bar{\mathcal D}(\mu, \sigma, r_2) := r_2^{-1} \mathcal D(\mu, \sigma, r_2) = \mathcal G_\eps(\mu,\sigma) + O(r_2) ,
	\]
	and fix $\sigma \in \Lambda$. If there exists a critical rate $\mu_c$ such that $\mathcal G_\eps(\mu_c, \sigma) = 0$ and $(\partial \mathcal G_\eps / \partial \mu)(\mu_c, \sigma) \neq 0$, then we obtain
	\[
	\bar{\mathcal D}(\mu_c, \sigma, 0) = 0 , \qquad 
	\frac{\partial \bar{\mathcal D}}{\partial r_2}(\mu_c, \sigma, 0) \neq 0 .
	\]
	Thus, by the implicit function theorem, there exists an $\eps_0 > 0$ and a function $\tilde \mu_c : (0,\eps_0) \to \R$ 
	such that $\mathcal D(\tilde \mu_c(r_2), \sigma, r_2) = 0$ for all 
	$r_2 \in (0,\eps_0)$. Using $r_2 = \eps$ gives the asymptotics
	\[
	\tilde \mu_c(\eps) = \mu_c + O(\eps) 
	\]
	as $\eps \to 0$, as required. This completes the proof of Assertion 1. Similar arguments can be used to prove Assertion 3, so we omit the details for brevity.

	\subsubsection{Proving Assertions 2 and 4}
	
	We first prove Assertion 2. Given Assumption \ref{ass:heteroclinic}(II) and using the results from \cite{Wechselberger2002} again, the signed distance function can be expanded in both $r_2$ and $\tilde \mu := \mu - \mu_c$:
	\begin{equation}
	\label{eq:distance_fn_r}
	\mathcal D(\tilde \mu, \sigma, r_2) = 
	\frac{\partial \mathcal D}{\partial \tilde \mu} (0,\sigma,0) \tilde \mu +
	\frac{\partial \mathcal D}{\partial r_2}(0,\sigma,0) r_2 + O \left( |(\tilde \mu, r_2)|^2 \right) ,
	\end{equation}
	where we now consider $\mathcal D$ to be a function of $\tilde \mu$, and the coefficient functions are given by the Melnikov integrals
	\begin{equation}
	\label{eq:Melnikov_integrals_r}
	\begin{split}
		\frac{\partial \mathcal D}{\partial \tilde \mu} (0, \sigma, 0) &= 
		\int_{-\infty}^\infty \left\langle \psi(t), \frac{\partial F}{\partial \mu} (\phi(t),\mu_c,\sigma,0) \right\rangle dt \\
		&=\int_{-\infty}^\infty \psi_1(t) t \gamma'(\mu_c t) \frac{\partial f}{\partial \gamma} (\varphi(t), \gamma(\mu_c t),0,\sigma,0) dt , \\
		&= \mathcal G_\mu(\sigma) ,
	\end{split}
	\end{equation}
	and
	\begin{equation}
	\label{eq:mathcal_G_eps_integrals_2}
	\frac{\partial \mathcal D}{\partial r_2}(0,\sigma,0) = 
	\mathcal G_\eps(\mu_c, \sigma) .
	\end{equation}
	By assumption, we have that $\mathcal G_\mu(\sigma) \neq 0$ ($\sigma$ is fixed here). Thus
	\[
	\mathcal D(0, \sigma, 0) = 0, \qquad 
	\frac{\partial \mathcal D}{\partial \tilde \mu}(0, \sigma, 0) \neq 0 ,
	\]
	so by the implicit function theorem there exists an $\eps_0 > 0$ and a function $\hat \mu_c : (0,\eps_0) \to \R$ 
	such that $\mathcal D(\hat \mu(r_2), \sigma, r_2) = 0$ for all 
	$r_2 \in (0,\eps_0)$. Using $r_2 = \eps$ gives the asymptotics
	\[
	\hat \mu_c(\eps) = \mu_c + O(\eps) 
	\]
	as $\eps \to 0$, as required. Equations \eqref{eq:distance_fn_r}, \eqref{eq:Melnikov_integrals_r}, \eqref{eq:mathcal_G_eps_integrals_2} and the implicit function theorem imply the form of the higher order correction in \eqref{eq:mu_c_asymptotics}. This completes the proof of Assertion 2.
	
	\

	The proof of Assertion 4 is similar. In this case, $\mu$ is fixed and the signed distance function can be expanded in $r_2$ and $\tilde \sigma := \sigma - \sigma_c$:
	\begin{equation}
	\label{eq:distance_fn_b}
	\mathcal D(\mu, \tilde \sigma, r_2) = 
	\frac{\partial \mathcal D}{\partial \tilde \sigma} (\mu,0,0) \tilde \sigma +
	\frac{\partial \mathcal D}{\partial r_2}(\mu,0,0) r_2 + O \left( |(\tilde \sigma, r_2)|^2 \right) ,
	\end{equation}
	where we now consider $\mathcal D$ to be a function of $\tilde \sigma$, and the coefficient functions are given by the Melnikov integrals
	\begin{equation}
	\label{eq:Melnikov_integrals_b}
	\begin{split}
		\frac{\partial \mathcal D}{\partial \tilde \sigma} (\mu, 0, 0) &= 
		\int_{-\infty}^\infty \left\langle \psi(t), \frac{\partial F}{\partial \sigma} (\phi(t),\mu,\sigma_c,0) \right\rangle dt \\
		&= \int_{-\infty}^\infty \psi_1(t) \frac{\partial f}{\partial \sigma} (\varphi(t), \gamma(\mu t),0,\sigma_c,0) dt , \\
		&= \mathcal G_\sigma(\mu) , \\
		\frac{\partial \mathcal D}{\partial r_2}(\mu,0,0) &= 
		\mathcal G_\eps(\mu, \sigma_c) .
	\end{split}
	\end{equation}
	The result follows similarly to the proof of Assertion 2 above via the implicit function theorem.
	\qed

	\subsection{Proof of Theorem \ref{thm:normal_hyperbolicity}}
	\label{sec:proof_of_thm_normal_hyperbolicity}
	
	We turn now to the proof of Theorem \ref{thm:normal_hyperbolicity}. Assertion 4 can be proven directly using Fenichel theory and asymptotic expansions, so we shall omit the details of the proof of this assertion. Moreover, Assumption \ref{ass:normal_hyperbolicity} is still in force, so Lemmas \ref{lem:K1_cm} and \ref{lem:K3_cm} are still true. Since the normal contractivity property in Assumption \ref{ass:tracking} implies that the eigenvalues of the $n \times n$ matrix $D_x f$ are bounded inside the left half plane for all $(x, s, \eps) \in \mathcal K$, the center manifolds $\mathcal M^\pm$ are attracting. The main task from here is to 
	prove that their `connection' via the inner transitionary layer is also attracting and normally hyperbolic for 
	for all $s_2$-values on \rev{the} arbitrarily large but bounded interval \rev{$s_2 \in \mathcal I = [-M, M]$} 
	in chart $K_2$. In order to prove the desired results, we use established results on the existence of normally hyperbolic manifolds, in particular \cite[Thm.~7.5]{Nipp2013}. 
	In order to apply the results in \cite{Nipp2013}, we need to
	\begin{itemize}
	\item[(i)] Prepare the equations with a cutoff procedure designed to control the `central directions';
	\item[(ii)] Show that the `auxiliary system' obtained in \rev{S}tep (i) satisfies the conditions that are necessary for the application of results in \cite{Nipp2013};
	\item[(iii)] Apply the results from \cite{Nipp2013} to the auxiliary system and use the output to complete the proof of Theorem \ref{thm:normal_hyperbolicity}.
	\end{itemize}
	We start with Step (i).

	\subsubsection{Step (i): Preparing the equations}
	\label{sec:conditions_1}
	
	Let $\mathcal X \subset \R^n$ be a non-empty, convex and open set, $\rev{\mathcal I} \subset \mathcal I' := (-M', M') \subset \R$ where $M' > \rev{M}$ so that $\rev{{{\mathcal I}}} \subset \mathcal I'$, and $\mathcal E := (0, \eps_0) \subset \R$. Now let $F : \mathcal X \times \mathcal I' \times \mathcal E \to \R^n$ denote the function defined by
	\[
	F(x,s,\eps) := f(x, \gamma(\mu s), \eps s, \eps) ,
	\]
	where $f$ is the same function considered in earlier sections except that we have suppressed the parameter dependence on $\sigma$ to simplify the notation (there is no added difficulty in the parameter-dependent case). Note that $F$ is in $C^k_b$ since $f$ is $C^k$ and $\mathcal X \times \mathcal I' \times \mathcal E$ is bounded. Using the above notation, we may rewrite the equations in $K_2$ (recall system \eqref{eq:K2_eqns}) as
	\begin{equation}
	\label{eq:K2_NS}
	\begin{split}
		x' &= F(x,s,\eps) , \\
		s' &= 1 ,
	\end{split}
	\end{equation}
	restricted to $(x,s,\eps) \in \mathcal X \times \mathcal I' \times \mathcal E$, where we have dropped the subscript on $s$ (which is written as $\hat s$ in Theorem \ref{thm:normal_hyperbolicity}) and used $r_2 = \eps$ in order to simplify notation.
	
	In order to apply the results in \cite{Nipp2013} directly to system \eqref{eq:K2_NS}, the dynamics would need to be (i) inflowing with respect to $\mathcal X$, and (ii) outflowing with respect to $\rev{\mathcal I}$. Unfortunately, the latter requirement can never be satisfied, since $s' = 1 > 0$ everywhere, implying inflowing dynamics at the left boundary $s = - \rev{M}$. In order to remedy this, we need to introduce an auxiliary system which coincides with system \eqref{eq:K2_NS} on the set $\rev{\mathcal K} := \mathcal X \times \rev{\mathcal I} \times \mathcal E$, but satisfies the relevant inflowing condition on a slightly larger set $\rev{\widetilde{\mathcal K}} := \mathcal X \times \rev{\widetilde{\mathcal I}} \times \mathcal E$, where \rev{$\widetilde{\mathcal I} := (- \widetilde M, \widetilde M)$} for a constant \rev{$\widetilde M \in (M, M')$} \rev{(}notice that \rev{${\mathcal I} \subset \widetilde{\mathcal I} \subset \mathcal I'$} and therefore \rev{${\mathcal K} \subset \widetilde{\mathcal K} \subset \mathcal K'$)}. \rev{Specifically, we consider} the auxiliary system
	\begin{equation}
	\label{eq:K2_NS_auxiliary}
	\begin{split}
		x' &= F(x,s,\eps) , \\
		s' &= 1 - \hat \delta \nu(s) ,
	\end{split}
	\end{equation}
	where $\hat \delta > 1$ is a constant close to but greater that unity, and $\nu : \R \to [0,1]$ is a 1-sided $C^\infty$ cutoff function of the form
	\[
	\nu(s) := 
	\begin{cases}
	1 , & s \leq - \rev{\widetilde M}, \\
	\e^{1/(\rev{\widetilde M} - \rev{M})} \left( 1 - \e^{-1/(s + \rev{\widetilde M})} \right) \e^{1/(s + \rev{M})}, & s \in (-\rev{\widetilde M}, - \rev{M}) , \\
	0 , & s \geq - \rev{M} .
	\end{cases}
	\]
	In particular, $\nu'(s) < 0$ on $s \in (-\rev{\widetilde M}, -\rev{M})$ and
	\[
	-\frac{1}{(\rev{\widetilde M} - \rev{M})^2} \leq \nu'(s) \leq 0 
	\]
	for all $s \in \R$. \rev{Importantly, the} auxiliary system \eqref{eq:K2_NS_auxiliary} is (i) outflowing with respect to \rev{$\widetilde{\mathcal I}$}, since $s'|_{s = - \rev{\widetilde M}} = 1 - \hat \delta < 0$ and $s'|_{s = \rev{\widetilde M}} = 1 > 0$, and (ii) coincident with system \eqref{eq:K2_NS} on $\rev{\mathcal K}$.
	

	\subsubsection{Step (ii): Checking conditions}
	\label{sec:conditions_2}
	
	In this section, we check the conditions for the applicability of the results in \cite{Nipp2013} to the auxiliary system \eqref{eq:K2_NS_auxiliary}.
	
	\begin{lemma}
	\label{lem:NS_HD}
	System \eqref{eq:K2_NS_auxiliary} satisfies \cite[Hypothesis HD]{Nipp2013}. In particular, the following properties are satisfied:
	\begin{itemize}
		\item[(a)] System \eqref{eq:K2_NS_auxiliary} is inflowing with respect to $\mathcal X$.
		\item[(b)] System \eqref{eq:K2_NS_auxiliary} is outflowing with respect to \rev{$\widetilde{\mathcal I}$}.
		\item[(c)] There exists non-negative constants $l_{12}$ and $l_{23}$ such that the bounds
		\[
		\| D_x F \| \leq l_{21} , \qquad \| D_\eps F \| \leq l_{23} , 
		\]
		hold uniformly on \rev{$\widetilde{\mathcal K}$}. Moreover, the logarithmic norm bounds
		\begin{equation}
			\label{eq:bounds_logarithmic_norm}
			\lim_{h \to 0} \frac{| 1 - h \hat \delta \nu'(s) | - 1}{h} \leq 0 , \qquad 
			\lim_{h \to 0} \frac{\| \mathbb I - h D_xF \| - 1}{h} \leq l_{22} < 0 ,
		\end{equation}
		hold uniformly on \rev{$\widetilde{\mathcal K}$}, for some contractivity constant $l_{22}$.
	\end{itemize}
	\end{lemma}
	
	\begin{proof}
	Property (a) follows from Assumption \ref{ass:tracking}(i), and property (b) is satisfied by construction (the cutoff term $\hat \delta \nu(s)$ was specifically designed for this purpose). The first two bounds in (c) are satisfied because the function $F$ is $C^k_b$ on \rev{$\widetilde{\mathcal K}$}, and the second logarithmic norm bound in \eqref{eq:bounds_logarithmic_norm} is true by Assumption \ref{ass:tracking}(ii). Finally, since $\nu'(s) \leq 0$ for all $s \in \R$, it follows that
	\[
	\lim_{h \to 0} \frac{| 1 - h \hat \delta \nu'(s) | - 1}{h} = 
	\hat \delta \nu'(s) \leq 0 ,
	\]
	as required.
	\end{proof}
	
	\begin{remark}
	\label{rem:NS_bounds}
	A number of the bounds in \cite{Nipp2013} are omitted in Lemma \ref{lem:NS_HD}, since they are automatically satisfied for systems in the general form \eqref{eq:K2_NS_auxiliary}. In particular, writing the $s'$ equation as $g(x,s,\eps) = 1 - \hat \delta \nu(s)$, we have
	\[
	l_{12} := \sup_{(x,s,\eps) \in \rev{\widetilde{\mathcal K}}} \left| \frac{\partial g}{\partial x} (x,s,\eps) \right| = 0 , \qquad 
	l_{13} := \sup_{(x,s,\eps) \in \rev{\widetilde{\mathcal K}}} \left| \frac{\partial g}{\partial \eps} (x,s,\eps) \right| = 0 .
	\]
	This leads to rather significant simplifications when it comes to checking the remaining conditions and deriving the main results.
	\end{remark}
	
	The next requirement is \cite[Hypothesis HDA]{Nipp2013}, i.e.~that the function $F(x^\ast, \cdot, \cdot) : \rev{\widetilde{\mathcal I}} \times \mathcal E \to \R^{n+1}$ is bounded for some $x^\ast \in \mathcal X$, but this follows immediately from the boundedness of $\rev{\widetilde{\mathcal I}} \times \mathcal E$ and the fact that $F$ is in $C^k_b$ on \rev{$\widetilde{\mathcal K}$}. The last remaining conditions are \cite[CD, CDA \& CDA(k)]{Nipp2013}. Here we simply note that due to the simplifications described in Remark \ref{rem:NS_bounds}, all three of these conditions reduce the requirement that $l_{22} < 0$. But, as already noted above, this is immediate from Assumption \ref{ass:tracking}(ii).

	\subsubsection{Step (iii): Applying known results and finishing the proof}
	
	Having verified the relevant conditions in Sections \ref{sec:conditions_1}-\ref{sec:conditions_2} above, we may now apply results from \cite{Nipp2013} to the auxiliary system \eqref{eq:K2_NS_auxiliary}.
	
	\begin{proposition}
	\label{prop:K2_auxiliary_manifold}
	Consider the auxiliary system \eqref{eq:K2_NS_auxiliary}. There exists a $C^k$ function $\omega_I : \rev{\widetilde{\mathcal I}} \times \mathcal E \to \mathcal X$ such that the following properties hold for all $\eps \in \mathcal E$:
	\begin{enumerate}
		\item The 1-dimensional manifold
		\[
		\mathcal W_I := \left\{ (\omega_I(s,\eps), s) : s \in \rev{\widetilde{\mathcal I}} \right\} 
		\]
		is negatively invariant, i.e.~if $(x,s) \in \mathcal X \times \rev{\widetilde{\mathcal I}}$, then $\Phi(x,s,t) \in \mathcal X \times \rev{\widetilde{\mathcal I}}$ for all $t \leq 0$, where $\Phi$ denotes the flow map induced by system \eqref{eq:K2_NS_auxiliary}. The function $\omega_I$ satisfies the invariance equation
		\[
		\left( 1 - \hat \delta \nu(s) \right) \frac{\partial \omega_I}{\partial s} (s, \eps) = 
		F(\omega_I(s,\eps), s, \eps) .
		\]
		\item The function $\omega_I$ is Lipschitz in both $s$ and $\eps$. More precisely, we have
		\[
		| \omega_I(s, \eps) - \omega_I(\tilde s, \eps) | \leq \frac{l_{21}}{|l_{22}|} , \qquad 
		| \omega_I(s, \eps) - \omega_I(s, \tilde \eps) | \leq \frac{l_{23}}{|l_{22}|} ,
		\]
		for all $s, \tilde s \in \rev{\widetilde{\mathcal I}}$ and $\eps, \tilde \eps \in \mathcal E$.
		\item The manifold $\mathcal W_I$ is uniformly attracting, i.e.~for all solutions $(x(t), s(t))$ with $(x(0), s(0)) \in \mathcal X \times \rev{\widetilde{\mathcal I}}$ we have
		\[
		\left| x(t) - \omega_I(s(t), \eps) \right| \leq 
		\e^{l_{22} t} \left| x(0) - \omega_I(s(0), \eps) \right| ,
		\]
		for all $t \geq 0$ such that $s(t) \in \rev{\widetilde{\mathcal I}}$.
	\end{enumerate}
	\end{proposition}
	
	\begin{proof}
	This follows directly from \cite[Thm.~7.5]{Nipp2013} after substituting the specific quantities for system \eqref{eq:K2_NS_auxiliary} derived in the lead up to the statement of the result.
	\end{proof}
	
	
	Proposition \ref{prop:K2_auxiliary_manifold} confirms the existence of a $C^k$-smooth, negatively invariant and uniformly attracting 1-dimensional manifold $\mathcal W$ in the auxiliary system \eqref{eq:K2_NS_auxiliary}. The idea now is to construct the function $\omega : (- \rho, \rho) \times (0, \eps_0) \to \R^n$ and the corresponding manifold $\mathcal W$ in Theorem \ref{thm:normal_hyperbolicity} using the center manifolds $\mathcal M_1^-$ and $\mathcal M_3^+$, together with the manifold $\mathcal W_I$ described in Proposition \ref{prop:K2_auxiliary_manifold}. More precisely, we consider the function $\omega : (- \rho, \rho) \times (0, \eps_0) \to \R^n$ defined by
	\begin{equation}
	\label{eq:omega_construction}
	\omega(s, \eps) := 
	\begin{cases}
		h_1^-(-s, -s / \eps) , & 0 < \eps / \rho \leq - s \leq \rho , \\
		\omega_I(s / \eps, \eps) , & s \in ( - \eps M, \eps M) , \\
		h_3^+(s, s / \eps) , & 0 < \eps / \rho \leq s \leq \rho , \\
	\end{cases}
	\end{equation}
	where $h_1^-$ ($h_3^+$) is the function which defines the center manifold $\mathcal M_1^-$ ($\mathcal M_3^+$) described in Lemma \ref{lem:K1_cm} (Lemma \ref{lem:K3_cm}), and $\omega_I$ is the function which defines the locally invariant manifold $\mathcal W_I$ described in Proposition \ref{prop:K2_auxiliary_manifold}. We set $\delta = \rho$ for simplicity, in which case Lemmas \ref{lem:K1_cm} and \ref{lem:K3_cm} apply as long as $\rho > 0$ is fixed sufficiently small. Fixing $M = \rho^{-1}$ ensures that $\omega$ is well-defined; the function $\omega_I$ can be -- and in the following will be -- chosen so that the Taylor expansions of $h_1^-(-s, -s/\eps)$ and $\omega_I(s/\eps, \eps)$ coincide at $s = - \eps \rho^{-1}$, and the Taylor expansions $h_3^+(s, s/\eps)$ and $\omega_I(s/\eps, \eps)$ coincide at $s = \eps \rho^{-1}$. 
	This is possible because all three manifolds are uniformly exponentially attracting and defined on overlapping domains if $M > \rho^{-1}$, which implies that their Taylor expansions coincide.
	
	We now proceed to verify the assertions and properties in Theorem \ref{thm:normal_hyperbolicity}. It suffices to prove Assertions 1-3 because, as noted above, Assertion 4 can be proven directly using Fenichel theory and series expansions.
	
	Assertion 1 follows by construction. In particular, the existence of the locally invariant manifold $\mathcal W$ in \eqref{eq:mathcal_W} follows from the definition of $\omega$ in \eqref{eq:omega_construction}, together with the local invariance of the manifolds $\mathcal S_\eps^\pm$ and $\mathcal W_I$ due to Lemmas \ref{lem:K1_cm}, \ref{lem:K3_cm} and Proposition \ref{prop:K2_auxiliary_manifold} respectively. The fact that $\omega_I$ satisfies the invariance equation \eqref{eq:omega_I_invariance} follows from Proposition \ref{prop:K2_auxiliary_manifold} Assertion 1 and the fact that $\nu(s) \equiv 0$ for all $\eps s \in \mathcal I$.
	
	Assertion 2 has already been proven in Proposition \ref{prop:K2_auxiliary_manifold}, and Assertion 3 follows from the exponential contractivity of the center manifolds $\mathcal M^-_1$, $\mathcal M^+_3$ and Proposition \ref{prop:K2_auxiliary_manifold} Assertion 3. This concludes the proof of Theorem \ref{thm:normal_hyperbolicity}.

	\section{Application to scalar asymptotically slow-fast equations}
	\label{sec:example}
	
	In the following we demonstrate our results in the context of scalar asymptotically slow-fast equations of the form
	\begin{equation}
	\label{eq:scalar_asf_eqn}
	x' = \gamma(\mu t) F_+(x) + \left( 1 - \gamma(\mu t) \right) F_-(x) + g(\gamma(\mu t), \eps t, \sigma, \eps) ,
	\end{equation}
	where the functions $F_\pm$ and $g$ are smooth and will be specified in each separate case below. When $g \equiv 0$, equation \eqref{eq:scalar_asf_eqn} is asymptotically autonomous with past (future) limiting system $x' = F_-(x)$ ($x' = F_+(x)$). The function $g$ can be interpreted as a parameter- and time-dependent forcing term. For simplicity, we shall will also assume that $\gamma$ is sigmoidal and satisfies
	\begin{equation}
	\label{eq:gamma_properties}
	\gamma'(z) > 0 , \qquad 
	\gamma(-z) = 1 - \gamma(z) , \qquad 
	\gamma'(z) = \gamma'(-z) ,
	\end{equation}
	for all $z \in \R$. For numerical purposes we shall use
	\begin{equation}
	\label{eq:gamma_numerics}
	\gamma(z) = \frac{1}{2} \left[ 1 + \frac{z}{\sqrt{1 + z^2}} \right] ,
	\end{equation}
	\rev{however, as we shall see, most of our results will be qualitatively independent of this choice} (although they do rely on the properties in \eqref{eq:gamma_properties}). 
	The corresponding autonomous system in $\R^2$, obtained by setting $s = \eps t$, is
	\begin{equation}
	\label{eq:scalar_asf_system}
	\begin{split}
		x' &= \gamma(\mu s / \eps) F_+(x) + \left( 1 - \gamma(\mu s / \eps) \right) F_-(x) + g(\gamma(\mu s / \eps), s, \sigma, \eps) , \\
		s' &= \eps ,
	\end{split}
	\end{equation}
	and the PWS system arising in the singular limit as $\eps \to 0$ is given by
	\begin{equation}
	\label{eq:scalar_asf_pws}
	\begin{pmatrix}
		x' \\
		s' 
	\end{pmatrix}
	=
	\begin{cases}
		\begin{pmatrix}
			F_-(x) + g(0, s, \sigma, 0) \\
			0 
		\end{pmatrix}
		& \quad \text{for} \quad s < 0 , \\
		\begin{pmatrix}
			F_+(x) + g(1, s, \sigma, 0) \\
			0
		\end{pmatrix}
		& \quad \text{for} \quad s > 0 ,
	\end{cases}
	\end{equation}
	which is (generically) discontinuous along the switching manifold $\Sigma = \{s = 0\}$. The 1-dimensional critical manifolds \rev{$\mathcal C^\pm$}, when they exist, are unions of curves defined by
	\begin{equation*}
	\begin{split}
		\rev{\mathcal C^-} &= \left\{ (x,s) \in \R \times \R_- : F_-(x) + g(0, s, \sigma, 0) = 0 \right\} , \\
		\rev{\mathcal C^+} &= \left\{ (x,s) \in \R \times \R_+ : F_+(x) + g(0, s, \sigma, 0) = 0 \right\} .
	\end{split}
	\end{equation*}
	We shall consider two different examples: One featuring bifurcation and rate induced transitions, and another featuring a persistent normally hyperbolic manifold.

	\subsection{Rate and bifurcation induced transitions}
	\label{sec:example_tipping}
	
	We consider system \eqref{eq:scalar_asf_eqn} with
	\begin{equation}
	\label{eq:Fg_tipping}
	F_-(x) = -x , \quad 
	F_+(x) = x(1 - x^2) , \quad 
	g(\gamma(\mu t), \eps t, \sigma, \eps) = A \sin(\eps t) + \eps \left(\sigma - \gamma(\mu t)^2 \right) ,
	\end{equation}
	where $A > 0$ is a parameter which determines the amplitude of the (slow) time-periodic forcing. In the following, we rewrite the forcing term as
	\[
	g(\gamma(\mu s / \eps), s, \sigma, \eps) = A \sin(s) + \eps \left(\sigma - \gamma(\mu s / \eps)^2 \right) ,
	\]
	and work with system \eqref{eq:scalar_asf_system}. The critical manifold $\rev{\mathcal C^-}$ contains a single branch
	\begin{equation}
	\label{eq:mathcal_S_m}
	\mathcal S^- = \left\{ (x, s) \in \R \times \R_- : x = A \sin(s) \right\} ,
	\end{equation}
	which (one can check) is normally hyperbolic and attracting for all $s < 0$, and the hyperbolicity condition at $p_- : (0,0)$ in Assumption \ref{ass:normal_hyperbolicity} (recall that $p_\pm = \mathcal S^\pm \cap \Sigma$) is satisfied with $D_x f_-(0,0,\sigma) = D_x F(0) = -1$. The critical manifold $\rev{\mathcal C^+}$ has either 1, 2 or 3 branches, depending on the number of real solutions to the equation $x (1 - x^2) + A \sin (s) = 0$. For our purposes, it suffices to note that if $s \in [0,\rho]$ and $\rho > 0$ is fixed sufficiently small, then $\rev{\mathcal C^+}|_{s \in [0,\rho]}$ is comprised of an upper, middle and lower branches of the form
	\begin{equation}
	\label{eq:mathcal_S_p}
	\begin{split}
		\rev{\mathcal C^+_u} &= \left\{ (h^+_u(s), s) : s \in [0,\rho] \right\} , \\
		\mathcal S^+ &= \left\{ (h^+(s), s) : s \in [0,\rho] \right\} , \\
		\rev{\mathcal C^+_l} &= \left\{ (h^+_l(s), s) : s \in [0,\rho] \right\} , 
	\end{split}
	\end{equation}
	respectively, where the functions $h^+_u$, $h^+$ and $h^+_l$ satisfy
	\[
	h^+_u(s) = 1 + O(s) , \qquad 
	h^+(s) = O(s) , \qquad
	h^+_l(s) = -1 + O(s) ,
	\]
	as $s \to 0^+$; see Figure \ref{fig:b_tipping}. Assuming again that $\rho > 0$ is sufficiently small, direct calculations show that $\rev{\mathcal C^+_u}$ and $\rev{\mathcal C^+_l}$ are normally hyperbolic and attracting for all $s \in (0,\rho]$, while $\mathcal S^+$ is normally hyperbolic and repelling for all $s \in (0,\rho]$. We also have that $p_+ : (0,0)$, and $D_x f_+(0,0,\sigma) = D_x F_+(0) = 1$. Thus, Assumption \ref{ass:normal_hyperbolicity} is satisfied and, by Fenichel theory, compact submanifolds of $\rev{\mathcal C^-} = \mathcal S^-$ and $\rev{\mathcal C^+ = \mathcal C^+_l \cup \mathcal S^+ \cup \mathcal C^+_u}$ which are bounded away from $\Sigma$ perturb to $O(\eps)$-close slow manifolds. As in earlier sections, we denote the corresponding slow manifolds by appending the subscript $\eps$.
	
	\begin{figure}[t!]
	\centering
	\includegraphics[scale=0.85]{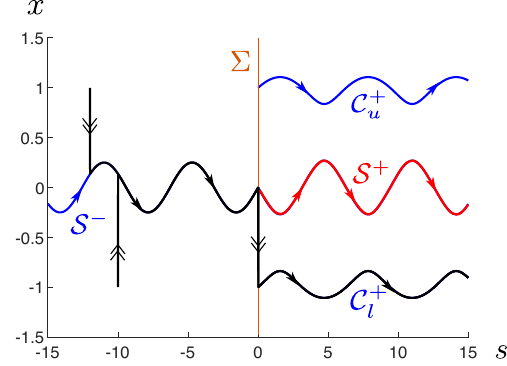}
	\ \ 
	\includegraphics[scale=0.85]{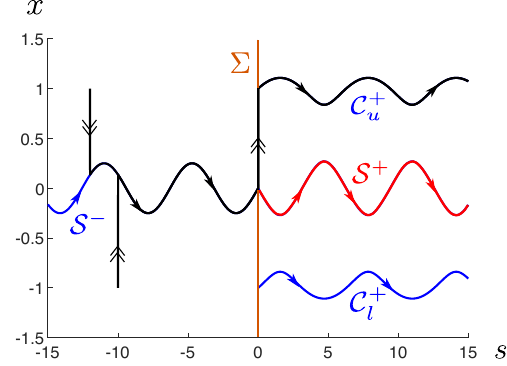}
	\caption{Bifurcation induced transition in system \eqref{eq:scalar_asf_system} with \eqref{eq:Fg_tipping}, $A = 1/4$ and $\mu = 1$. In black, we show solutions with different initial conditions, computed numerically using a stiff ODE solver (ode15s) in Matlab with $\eps = 0.001$. The left panel shows solutions for $\sigma = 0.36$, and the right panel shows solutions for $\sigma = 0.37$. We have also plotted the critical manifolds $\rev{\mathcal C^-} = \mathcal S^-$, $\rev{\mathcal C^+ = \mathcal C^+_l \cup \mathcal S^+ \cup \mathcal C^+_u}$ and the switching manifold $\Sigma$ identified in the singular limit analysis of the limiting slow-fast PWS system \eqref{eq:scalar_asf_pws}. The color conventions are consistent with earlier figures and, as before, single (double) arrows represent slow (fast) evolution in time. Solutions in the left panel are exponentially close to $\rev{\mathcal C^+_{l,\eps}}$ ($\rev{\mathcal C^+_{u,\eps}}$) when they leave a neighbourhood of $\Sigma$, showing that the forward time attractor has changed due to variation in $\sigma$ over a critical value of $\sigma_c(1) \approx 0.3636$ predicted by Theorem \ref{thm:tipping}.}
	\label{fig:b_tipping}
	\end{figure}
	
	We now use Assertions 1 and 3 of Theorem \ref{thm:tipping} in order to identify rate and bifurcation induced transitions in which the forward extension of $\mathcal S^-_\eps$ `jumps' from \rev{$\mathcal C_{l,\eps}^+$} to \rev{$\mathcal C_{u,\eps}^+$} as $\mu$ and $\sigma$ are varied over an $\eps$-dependent neighbourhoods about critical values $\mu_c$ and $\sigma_c$ respectively.
	%
	%
	Note that Assumption \ref{ass:heteroclinic}(I) is satisfied, with $\rev{\phi(t)} = (0,t)$. This leads to the following exponentially decaying solution to the adjoint variational equation \eqref{eq:adjoint_variational_eqn}:
	\begin{equation}
		\label{eq:psi_1}
		\psi_1(t, \mu) = \exp \left[ \int_1^t \left(1 - 2 \gamma(\mu \xi) \right) d \xi  \right] .
	\end{equation}
	One can verify with direct calculations using the properties in \eqref{eq:gamma_properties} that $\psi_1(t)$ is even, bounded, and strictly positive for all $t \in \R$. In order to identify critical rate and bifurcation values associated with a transition, we use \eqref{eq:psi_1} to look for roots of the function $\mathcal G_\eps(\mu, \sigma)$ defined in \eqref{eq:mathcal_G_eps}. Direct evaluation and a little algebraic manipulation leads to
	\begin{equation}
		\label{eq:mathcal_G_eps_integrals}
		\begin{split}
			\mathcal G_\eps(\mu, \sigma) &= A \int_{-\infty}^{\infty} \psi_1(t, \mu) t dt + \sigma \int_{-\infty}^{\infty} \psi_1(t, \mu) dt - \int_{-\infty}^{\infty} \psi_1(t, \mu) \gamma(\mu t)^2 dt \\
			&= \sigma \int_{-\infty}^{\infty} \psi_1(t, \mu) dt - \int_{-\infty}^{\infty} \psi_1(t, \mu) \gamma(\mu t)^2 dt ,
		\end{split}
	\end{equation}
	where we used the fact that the map $t \mapsto \psi_1(t, \mu) t$ is odd for each fixed $\mu$ (since $t \mapsto \psi_1(t, \mu)$ is even) in order to eliminate the left-most integral in the first equality. Since both integrals are strictly positive, equation \eqref{eq:mathcal_G_eps_integrals} must have at least one root if $\sigma$ is sufficiently small and positive. This occurs for $\sigma = \sigma_c(\mu)$, where
	\begin{equation}
		\label{eq:sigma_c_pitchfork}
		\sigma_c(\mu) := \frac{\int_{-\infty}^{\infty} \psi_1(t, \mu) \gamma(\mu t)^2 dt}{\int_{-\infty}^{\infty} \psi_1(t, \mu) dt} .
	\end{equation}
	Since
	\[
	\frac{\partial \mathcal G_\eps}{\partial \sigma} (\mu, \sigma_c(\mu)) = \int_{-\infty}^{\infty} \psi_1(t, \mu) dt > 0 
	\]
	for all fixed $\mu > 0$, Theorem \ref{thm:tipping} Assertion 3 implies the desired intersection for
	\begin{equation}
		\label{eq:tilde_sigma}
		\tilde \sigma(\mu, \eps) = \sigma_c(\mu) + O(\eps) 
	\end{equation}
	as $\eps \to 0$. We verified this result numerically for parameter values $\mu = 1$, $A = 1/4$, $\eps = 0.001$ and two different values of $\sigma$. This is shown in Figure \ref{fig:b_tipping}, which shows that a bifurcation induced transition occurs near a critical value \rev{of} $\sigma$ which lies between $0.36$ and $0.37$. This agrees with the theoretical prediction of $\tilde \sigma_c(1,0.001) \approx \sigma_c(1) \approx 0.3636$ obtained using \eqref{eq:sigma_c_pitchfork}.
	
	\
	
	\begin{figure}[t!]
		\centering
		\includegraphics[scale=0.4]{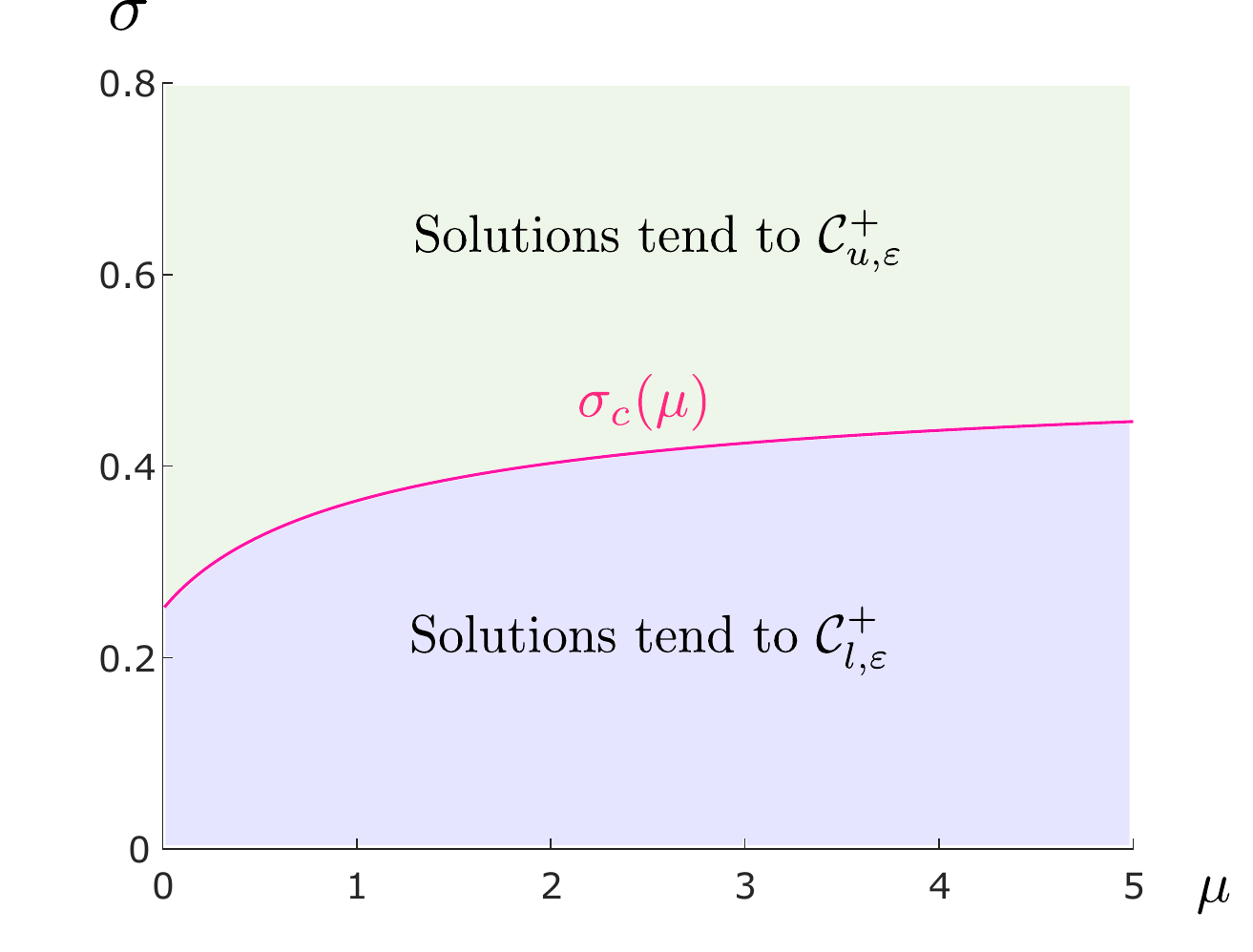}
		\caption{Critical curve defined by the function $\sigma_c(\mu)$ in \eqref{eq:sigma_c_pitchfork} the $(\mu, \sigma)$-parameter plane (magenta). For each $(\mu, \sigma)$ fixed in the blue region below (green region above) the curve, solutions of system \eqref{eq:scalar_asf_system} with \eqref{eq:Fg_tipping} are exponentially close to the slow manifold $\rev{\mathcal C^+_{l,\eps}}$ ($\rev{\mathcal C^+_{u,\eps}}$) when they leave a neighbourhood of the transitionary regime about $\Sigma$. To leading order as $\eps \to 0$, the bifurcation induced transition in Figure \ref{fig:b_tipping} corresponds to a transversal intersection of the curve defined by $\sigma = \sigma_c(\mu)$ with the vertical line defined by $\mu = 1$. A rate induced transition corresponds to a transversal intersection with a horizontal line defined by a constant $\sigma$ value.}
		\label{fig:2par_plane}
	\end{figure}	
	
	In order to identify a rate induced transition, it suffices to show that
	\[
	\frac{\partial \mathcal G_\eps}{\partial \mu} (\mu_c, \sigma) \neq 0
	\]
	for any fixed choice of $\sigma > 0$ which is small enough to ensure that equation \eqref{eq:mathcal_G_eps_integrals} has a root for some $\mu = \mu_c$. We were unable to verify this condition analytically, however one can easily verify it numerically for particular parameter values. In fact, we claim a stronger result, namely that
	\[
	\frac{\partial \mathcal G_\eps}{\partial \mu} (\mu, \sigma_c(\mu)) \neq 0
	\]
	for an entire interval of $\mu$-values (if not all $\mu > 0$), based on the numerically computed graph of $\sigma = \sigma_c(\mu)$ in the $(\mu,\sigma)$-plane shown in Figure \ref{fig:2par_plane}. Rate induced transitions occur as one `crosses' this curve by varying $\mu$ along lines of constant $\sigma$. For example, if $\sigma = 0.4$, then the diagram predicts a rate induced transition at a critical rate of $\mu_c \approx 1.88$. This can be verified by simulations which show behaviour that is qualitatively similar to the bifurcation induced transition shown in Figure \ref{fig:b_tipping}.
	
	\begin{remark}
		\rev{The preceding analysis shows that the transition depicted in Figure \ref{fig:b_tipping} occurs if $\sigma$ and/or $\mu$ are varied over a path in $(\sigma, \mu)$-space which transversally intersects the graph defined by $\sigma = \sigma_c(\mu)$ in Figure \ref{fig:2par_plane} as $\eps \to 0$. In this sense, the qualitative effect on the dynamics does not depend on whether one varies $\mu$ as opposed to $\sigma$, i.e.~it does not depend on whether the transition is `bifurcation induced' or `rate induced'.}
	\end{remark}
	
	\begin{remark}
		\rev{Our results on bifurcation-induced transitions are qualitatively independent of the choice ramp/regularisation function $\gamma$ insofar as the arguments we used to prove the existence of a bifurcation-induced transition when $\sigma = \tilde \sigma(\mu, \eps)$ hold for any ramp/regularisation function $\gamma$ satisfying \eqref{eq:gamma_properties}. We conjecture that rate induced transitions also occur for any $\gamma$ satisfying \eqref{eq:gamma_properties}, however, the results provided above are numerical and therefore depend on the particular choice of $\gamma$ in \eqref{eq:gamma_numerics}.}
	\end{remark}
	
	\begin{remark}
		The arguments presented in \rev{this} section are sufficient to prove the existence of canard-like intersections close to the parameter space curve defined in \eqref{eq:tilde_sigma}. It is also possible to prove that solutions leave a neighbourhood of $\Sigma$ exponentially close to \rev{$\mathcal C_{l,\eps}^+$ ($\mathcal C_{u,\eps}^+$)} as $\eps \to 0$, as seen in Figure \ref{fig:b_tipping}, if the point $(\mu,\sigma)$ is fixed below (above) the curve defined by $\sigma = \sigma_c(\mu)$ in Figure \ref{fig:2par_plane}. The proof, which is omitted for brevity, is similar to the proof of \cite[Thm.~4.1]{Krupa2001c}, which describes the exchange of stability under parameter variation near a pitchfork singularity in planar slow-fast systems.
	\end{remark}

	\subsection{Tracking}
	
	We now consider system \eqref{eq:scalar_asf_system} with
	\begin{equation}
		\label{eq:Fg_tracking}
		F_-(x) = -x , \qquad 
		F_+(x) = -x(1 - x^2) , \qquad 
		g(\gamma(\mu t), \eps t, \sigma, \eps) = A \sin(s) + \eps .
	\end{equation}
	This defines a system which is similar to that considered in Section \ref{sec:example_tipping}, except with a different sign in front of $F_+$ and a simpler forcing function $g$. In this case, the PWS system \eqref{eq:scalar_asf_pws} obtained in the singular limit has the same normally hyperbolic and attracting critical manifold $\mathcal S^-$, but the stability properties of the critical manifold on $s > 0$ are different due to the change in sign in front of $F_+$. Nevertheless, the structure is similar, and there are three distinct branches for $s \in [0,\rho]$ as long as $\rho > 0$ is fixed sufficiently small. We want to show that the slow manifold $\mathcal S^-_\eps$ extends through the transitionary regime and connects to the slow manifold $\mathcal S^+_\eps$ which perturbs from the middle branch of $\rev{\mathcal C^+}$, which we denote again by $\mathcal S^+$. Notice that $p_- : (0,0)$ and $p_+ : (0,0)$, similarly to the previous section, but (in contrast to the previous section) $D_x f_+(p_+) = D_x F_+(0) = - 1$, which shows that $\mathcal S^+$ is locally normally hyperbolic and attracting. Assumption \ref{ass:normal_hyperbolicity} is still satisfied.\footnote{\rev{One can also verify Assumption \ref{ass:heteroclinic} for this system, however, Theorem \ref{thm:tipping} does not apply because there is no positive eigenvalue of $D_x f_+(p_+)$.}}
	
	We obtain the following result.
	
	\begin{proposition}
		\label{prop:tracking}
		Consider system \eqref{eq:scalar_asf_system} with \eqref{eq:Fg_tracking}. There exists an $\eps_0 > 0$ and a $C^k$ function $\omega : [- \rho, \rho ] \times (0,\eps_0) \to \R$ such that the following assertions are true for all $A < 3/8 - \eps_0$:
		\begin{enumerate}
			\item The 1-dimensional manifold
			\[
			\mathcal W := \left\{ (\omega(s,\eps), s) : s \in [-\rho, \rho] \right\}
			\]
			is locally invariant under the flow. Moreover, the restricted function $\omega_I := \omega|_{s \in I}$ satisfies the invariance equation
			\[
			\frac{\partial \omega_I}{\partial s} (s, \eps) = 
			f(\omega_I(s,\eps), \gamma(\mu s), s, \eps) ,
			\]
			where $I = [-\eps M, \eps M]$ is the interval defined before the statement of Theorem \ref{thm:normal_hyperbolicity}.
			\item $\omega_I$ is Lipschitz in both $s$ and $\eps$. More precisely, we have
			\[
			| \omega_I(s, \eps) - \omega_I(\tilde s, \eps) | \leq 4 , \qquad 
			| \omega_I(s, \eps) - \omega_I(s, \tilde \eps) | \leq 4 ,
			\]
			for all $s, \tilde s \in I$ and $\eps, \tilde \eps \in (0,\eps_0)$.
			\item $\mathcal W$ is uniformly attracting. \rev{In particular,} for all solutions $(x(t), s(t))$ with $(x(0), s(0)) \in \mathcal X \times (-\rho, \rho)$, we have
			\[
			\left| x(t) - \omega(s(t), \eps) \right| \leq 
			\e^{- t / 4} \left| x(0) - \omega(s(0), \eps) \right| ,
			\]
			for all $t \geq 0$ such that $s(t) = \eps t + s(0) \in (-\rho, \rho)$.
			\item For each fixed $c \in (0,\rho)$, $\mathcal W$ is $O(\eps)$-close to $\mathcal S^-$ ($\mathcal S^+$) on $s \in [-\rho, -c]$ ($s \in [c,\rho])$. More precisely, the restricted functions $\omega_O^- := \omega|_{s \in [-\rho,-c]}$ and $\omega_O^+ := \omega|_{s \in [c, \rho]}$ have asymptotics
			\[
			\omega_O^\pm(s,\eps) = h^\pm(s) + O(\eps) 
			\]
			as $\eps \to 0$, where $h^+(s) = O(s)$ as $s \to 0^+$ and $h^-(s) = A \sin(s)$.
		\end{enumerate}
	\end{proposition}
	
	\begin{proof}
		In order to use Theorem \ref{thm:normal_hyperbolicity}, we need to choose a neighbourhood $\mathcal X$. For simplicity we choose $\mathcal X = (-1/2, 1/2)$, however we note that different choices may lead to more optimal results; see again Remark \ref{rem:optimality}. With this choice, one can verify Assumption \ref{ass:tracking}(i) directly using the inequalities
		\[
		x'|_{x = \frac{1}{2}} < - \frac{3}{8} + A + \eps , \qquad 
		x'|_{x = -\frac{1}{2}} > \frac{3}{8} - A + \eps ,
		\]
		which imply that system \eqref{eq:scalar_asf_system} is (strictly) inflowing with respect to $\partial \mathcal X$ as long as $A < 3/8 - \eps_0$. In order to verify Assumption \ref{ass:tracking}(ii), we use the fact that
		\[
		D_x f(x, \gamma(\mu s / \eps), s, \eps) = -1 + 3 x^2 \gamma(\mu s / \eps) \in [-1, -1/4] 
		\]
		for all $x \in \mathcal X$ and $s \in \R$ in order to show that the logarithmic norm inequality in \eqref{eq:log_norm} is satisfied with $l_{22} = - 1/4$.
		
		The preceding arguments show that Assumption \ref{ass:tracking} is satisfied. This allows for the application of Theorem \ref{thm:normal_hyperbolicity}. Using the fact that $l_{21} = l_{23} = 1$, which follows from the definitions in \eqref{eq:l_ij} and the form of the equations in \eqref{eq:scalar_asf_system} and \eqref{eq:Fg_tracking}, Assertions 1-4 follow directly from Assertions 1-4 in Theorem \ref{thm:normal_hyperbolicity} respectively.
	\end{proof}
	
	\begin{remark}
		\rev{The proof of Proposition \ref{prop:tracking} does not depend on the particular form of $\gamma$ in \eqref{eq:gamma_numerics}, i.e.~the statements hold for any ramp/regularisation function $\gamma$ satisfying \eqref{eq:gamma_properties}.}
	\end{remark}

	\section{Summary and outlook}
	\label{sec:summary_and_outlook}
	
	The increased attention to the study of bifurcation and rate induced tipping in non-autonomous dynamical systems may be attributed to prevalence of `tipping phenomena' in important complex systems including the Earth's climate system \cite{Ashwin2012,Lenton2011,Lenton2008,Scheffer2020,Wieczorek2011}, and detailed mathematical approaches to the study of bifurcation and rate induced transitions have been undertaken in the context of asymptotically autonomous systems of the general form \eqref{eq:asymptotically_autonomous_system}. A major analytical advantage of restricting to asymptotically autonomous systems is the availability of classical dynamical systems methods, \rev{following a suitable compactification procedure which may be applied to an extended autonomous system} \cite{Chen2023,Wieczorek2023,Wieczorek2021}. Unfortunately, despite significant analytical advantages, these approaches cannot account for non-trivial dependence on time as $t \to \pm \infty$, which may be restrictive in applications. In contrast, the class of `truly non-autonomous' systems of the form \eqref{eq:nonautonomous_system} is much larger, but systems within this class can be extremely challenging to analyse, with existing works showing that detailed analytical treatment can be incredibly intricate even for scalar equations with a particular structure \cite{Duenas2023c,Duenas2023b,Duenas2023,Longo2022,Longo2023,Longo2021}.
	
	In this work, we developed and applied a mathematical framework for the study of asymptotically slow-fast systems in the general form \eqref{eq:asymptotically_sf_systems}. These systems form a class which is strictly larger than the class of asymptotically autonomous systems \eqref{eq:asymptotically_autonomous_system}, but strictly smaller than the class of non-autonomous systems \eqref{eq:nonautonomous_system}. In Section \ref{sec:setup_and_assumptions}, we showed that the past and future subsystems of these systems are slow-fast, with 1-dimensional critical manifolds \rev{$\mathcal C^\pm$}. In addition to the slow-fast structure in the `outer regimes' where $|s| > 0$, the limit as $\eps \to 0$ leads to the loss of smoothness as the entire inner/transitionary regime for which $s = O(\eps)$ is condensed into a codimension-1 switching manifold $\Sigma = \{ s = 0 \}$. In Section \ref{sec:geometric_blow-up}, we showed that the loss of smoothness along $\Sigma$ can be resolved via geometric blow-up. A primary objective of this work has been to show that the blow-up method provides a natural geometric approach to the study of asymptotically slow-fast systems. In particular, it provides a rigorous and geometric way of connecting the limiting dynamics on the `compactified' inner regime, which is governed by an asymptotically autonomous system that is constrained to the surface of the blow-up cylinder, to the limiting slow-fast systems which govern the dynamics in the outer regimes with $s < 0$ and $s > 0$. It is also worthy to note that the blow-up procedure itself does not rely on any of the Assumptions \ref{ass:normal_hyperbolicity}-\ref{ass:tracking}.
	
	In order to demonstrate the utility of the formalism, we used it to prove two main results, both of which were stated in Section \ref{sec:main_results}. The first of these, Theorem \ref{thm:tipping}, provides sufficient conditions for canard-type intersection of forward and backward extensions of 1-dimensional slow manifolds $\mathcal S^-_\eps$ and $\mathcal S^+_\eps$, for locally unique parameter values (the rate parameter $\mu$ in the case of Assertions 1-2, and the bifurcation parameter $\sigma$ in the case of Assertions 3-4). For scalar asymptotically slow-fast equations these intersections form separatrices, and the fact that they break regularly under parameter variation means that we have identified a mechanism for rate and bifurcation induced transitions. Importantly, Theorem \ref{thm:tipping} also provides a means for \textit{calculating} the critical values, at least to leading order in $\eps$. In the case of Assertions 1 and 3, a critical value $\mu_c$ or $\sigma_c$ corresponds to a simple zero of the Melnikov function \eqref{eq:mathcal_G_eps}. In the case of Assertion 2 or 4 a critical value $\mu_c$ or $\sigma_c$ corresponds to the existence of a particular solution of the inner limiting problem \eqref{eq:limit_problem} which satisfies Assumption \ref{ass:heteroclinic}(II) or Assumption \ref{ass:heteroclinic}(III) respectively. Our second main result is Theorem \ref{thm:normal_hyperbolicity}, which provides sufficient conditions for the persistence of an attracting normally hyperbolic manifold through the transitionary regime. In essence, this result provides conditions under which the extension of an attracting slow manifold in the left-outer regime can be `connected' to an attracting slow manifold in the right-outer regime, while maintaining uniform attractivity in the transitionary regime. The blow-up analysis in Section \ref{sec:geometric_blow-up} played a crucial role in the proof of both Theorems \ref{thm:tipping} and \ref{thm:normal_hyperbolicity}, which appeared in Sections \ref{sec:proof_of_thm_tipping} and \ref{sec:proof_of_thm_normal_hyperbolicity} respectively. In addition to the blow-up analysis, we relied quite heavily on the variant of Melnikov theory developed for the study of heteroclinic-like intersections of invariant manifolds on non-compact domains in \cite{Wechselberger2002} to prove Theorem \ref{thm:tipping}, and general invariant manifold theory, in particular \rev{the} formulation in \cite{Nipp2013}, to prove Theorem \ref{thm:normal_hyperbolicity}. As with the blow-up approach, these are well-developed methods which can be applied in rather general contexts.
	
	Finally, we applied our results to examples of scalar asymptotically slow-fast systems in Section \ref{sec:example}. We used Theorem \ref{thm:tipping} to identify a bifurcation and rate induced transitions, and Theorem \ref{thm:normal_hyperbolicity} in order to verify and describe the persistence of an attracting and normally hyperbolic locally invariant manifold through the inner transitionary regime. 
	
	\
	
	The primary aim of this work has been to lay the foundations for a rigorous and geometric approach to the study of asymptotically slow-fast systems. Based on this foundation, we anticipate a lot of open problems and possibilities for future research. At this point, for example, we have not yet treated any applications (with the exception of the simple examples in Section \ref{sec:example}). We have also restricted our analysis to spatially homogeneous and deterministic systems, which rules out the possibility of spatially and noise induced transitions. The question of whether the formalism developed herein can be extended to the study of asymptotically slow-fast SDE's and asymptotically slow-fast PDE's is entirely open, albeit non-trivial, due in part to the fact that the geometric blow-up method needs further development in these areas (see however \cite{Arcidiacono2020} and \cite{Engel2022,Engel2020,Jelbart2022,Jelbart2023,Zacharis2023} for recent progress in the context of asymptotically autonomous systems and PDEs respectively). At present, these and other open problems remain for future work.

	\section*{Acknowledgments}
	
	This work was \rev{partially} funded by the SFB/TRR 109 Discretization and Geometry in Dynamics, i.e.~Deutsche Forschungsgemeinschaft (DFG—German Research Foundation) - Project-ID 195170736 - TRR109\rev{, and partially funded via a European Union Marie Skłodowska-Curie Postoctoral Fellowship - Grant Agreement ID 101103827.} I would also like to thank Iacopo Longo for helpful and interesting discussions on critical transitions in non-autonomous systems.

	\bibliographystyle{siam}
	\bibliography{NADs}

\end{document}